\documentclass[11pt]{article}
\usepackage{amssymb,amsmath,accents}
\usepackage{amscd}
\usepackage{amsfonts,amsthm,mathrsfs}
\usepackage{setspace}
\usepackage{cases,empheq}
\usepackage{subcaption}
\usepackage{diagbox}
\usepackage[hyphens,allowmove]{url}
\usepackage{hyperref}
\usepackage{cleveref}
\usepackage{float}

\numberwithin{equation}{section}
\newtheorem{thm}{Theorem}
\newtheorem{cor}[thm]{Corollary}
\newtheorem{lem}[thm]{Lemma}
\newtheorem{prop}[thm]{Proposition}

\newtheorem{definition}[thm]{Definition}
\newtheorem{rem}[thm]{Remark}

\begin{document}

\title{On existence of a collapsed bubble with surface tension \\ in viscous incompressible fluid}
\bigskip
\author{Yoshikazu Giga$^1$ and Zhongyang Gu$^2$ \vspace{0.7em}\\ 
\centerline{{\small $^1$Graduate School of Mathematical Sciences}} \\
\centerline{{\small The University of Tokyo}} \\
\centerline{{\small 3-8-1 Komaba, Meguro-ku, Tokyo 153-8914, Japan}} \vspace{0.7em}\\
\centerline{{\small $^2$School of Mathematical Sciences}} \\
\centerline{{\small Shenzhen University}} \\
\centerline{{\small Nanshan district 518060, Shenzhen, P.~R.~China}}}
\date{}
\maketitle
\thispagestyle{empty}
%
%%%%%%%%%%%%%%
\begin{abstract}
We consider the one-phase free boundary problem for the incompressible Navier-Stokes equations in $\mathbb{R}^d$ ($d\ge2$).
The surface tension is taken into account.
 The initial domain, which is the outside a bubble, is an exterior domain.
 We prove that there exists a bubble evolving by this free boundary problem which collapses in a finite time without blowing up of principal curvatures of its boundary.
 In other words, what is called a splash singularity is formed in a finite time.
 This type of result also is valid for a bounded initial domain.
To construct such an example, we introduce the notion of a domain with $\delta$-wing which is a flat Riemannian manifold that is not embedded in $\mathbb{R}^d$, but it covers the $\delta$-neighborhood of the original domain whose boundary is self-intersected.
\end{abstract}
\maketitle
\thispagestyle{empty}

%%%%%%%%%%%%%%%%%%%%%%%%%%%%%%%%%%%%%%%%%%%%%%%%%
\section{Introduction} \label{SI} % Section 1

We consider the one-phase free boundary problem for the Navier-Stokes equations with surface tension of the form
\begin{equation}
\begin{aligned} \label{ENS}
    \partial_t u+(u\cdot\nabla)u-\operatorname{div}\mathbb{S}(u,\varpi) &=0
    &&\quad\text{in}\quad \Omega(t), \\
    \operatorname{div}u &=0
    &&\quad\text{in}\quad \Omega(t), \\
    \mathbb{S}(u,\varpi)\mathbf{n} &=\sigma\kappa\mathbf{n}
    &&\quad\text{on}\quad \partial\Omega(t), \\
    u\cdot\mathbf{n}  &=V
    &&\quad\text{on}\quad \partial\Omega(t),
\end{aligned}
\end{equation}
where $\mathbb{S}$ denotes the stress tensor defined by
\[
    \mathbb{S}(u,\varpi)=-\varpi I+2\mu\mathbb{D}(u), \quad
    \mathbb{D}(u)=\frac12 \left(\nabla u+(\nabla u)^T\right)
\]
and $\mathbf{n}$ denotes the outward unit normal of $\partial\Omega$; 
$\kappa$ denotes $(d-1)$ times mean curvature of $\partial\Omega$; 
$V$ denotes the normal velocity of $\partial\Omega$ in the direction of $\mathbf{n}$.
Here we assume that viscosity coefficient $\mu$ and surface tension coefficient $\sigma$ are positive constants.
 The condition \eqref{ENS}$_4$ is called a kinematic boundary condition which requires that $\partial\Omega(t)$ moves by the fluid velocity $u$.
 The condition \eqref{ENS}$_3$, usually known as the free boundary condition, is the force balance, i.e., the force from the fluid is proportional to the surface tension.

The problem is supplemented with initial condition
\begin{equation} \label{EIn}
    \Omega(0)=\Omega_0, \quad
    u(x,0)=u_0(x).
\end{equation}
We are interested in the case that $\Omega_0$ is a domain in $\mathbb{R}^d$ ($d\ge2$) whose boundary $\Gamma=\partial\Omega$ in $\mathbb{R}^d$ is compact to simplify the presentation.
 In other words, $\Omega_0$ is either a bounded or an exterior domain.
 For a bubble we mean that it is a complement of an exterior domain.
 The unknowns for \eqref{ENS}, \eqref{EIn} are velocity $u$, pressure $\varpi$ and domain $\Omega(t)$.
 We need to assume that $u$ decays at space infinity for an exterior domain.

We consider a domain $\Omega$ in $\mathbb{R}^d$ which lies locally one side of a boundary but the boundary is not embedded in $\mathbb{R}^d$ so that there is a self-intersection point.
Such a domain is often called a splash domain as in \cite{CS}; for a precise definition see Definition~\ref{DIm} in Section~\ref{SD}.
Our definition is slightly weaker than that of \cite{CS} because we do not assume the uniqueness of a self-intersection point.

The goal of this paper is to construct a family $\{\Omega_{0,\eta}\}$ of domains contained in $\Omega$ with smooth embedded boundary and initial velocity $u_{0,\eta}$ such that the solution $\Omega_\eta(t)$ of \eqref{ENS}, \eqref{EIn} with initial data $\Omega_{0,\eta}$ and $u_0$ ``collapses" near a self-intersection point of $\partial\Omega$ in a finite time.
 In other words, a self-intersection point appears on $\partial\Omega_\eta(t)$ and such ``singularity" is often called a splash singularity in the literature \cite{CS}.
 Our result is summarized in an informal way as follows.
 It characterizes the phenomenon of a bubble collapsing in finite time.
\begin{thm} \label{TMain}
Let $\Omega$ be a $C^3$ splash domain with compact boundary in $\mathbb{R}^d$.
 There is a non-decreasing sequence $\{\Omega_{0,\eta}\}_{\eta>0}$ of domains with $C^3$ embedded compact boundary in $\mathbb{R}^d$ ($d\ge2$) such that the following properties hold.
\begin{enumerate}
\item[(i)] As $\eta\to0$, $\Omega_{0,\eta}$ converges to a splash domain (in $C^3$ sense) contained in $\Omega$ and close to $\Omega$ in $\mathbb{R}^d$.
    Let $P$ be a self-intersection point of the boundary of $\Omega$. 
\item[(i\hspace{-0.1em}i)] There exists a velocity field $u_{0,\eta}\in C^1(\bar{\Omega}_{0,\eta})$ (decaying at space infinity in some sense) satisfying the compatibility condition
\begin{equation} \label{EComp}
    \operatorname{div}u_{0,\eta}=0\ \text{in}\ \Omega_{0,\eta},\quad
        \left(\mathbb{D}(u_{0,\eta})\mathbf{n}_{0,\eta}\right)_{\tan}=0\ \text{on}\ \partial\Omega_{0,\eta}
\end{equation}
    such that the boundary $\partial\Omega_\eta(t)$ of the solution \eqref{ENS}, \eqref{EIn} with $\Omega_0=\Omega_{0,\eta}$, $u_0=u_{0,\eta}$ ``collapses" in a finite time.
    More precisely, $\Omega_\eta(t)$ becomes a splash domain in a finite time whose boundary $\partial\Omega_\eta(t)$ has a self-intersection point near $P$.
    Furthermore, the solution vector field $u$ and the principal curvature of $\partial\Omega_\eta(t)$ is bounded as $t\uparrow T_c$, where $T_c$ is the collapsing time.
\end{enumerate}
\end{thm}

Here $(w)_{\tan}$ of a vector field $w$ on $\partial\Omega_{0,\eta}$ denotes the tangential part, i.e., $(w)_{\tan}=(I-\mathbf{n}_{0,\eta}\otimes\mathbf{n}_{0,\eta})w$, where $\mathbf{n}_{0,\eta}$ denotes the unit exterior vector of $\partial\Omega_{0,\eta}$ and $I$ denotes the identity matrix.
 By a collapsing time $T_c$ we mean that
\[
	T_c := \sup \left\{t \bigm|
    \partial\Omega_{0,\eta}(\tau)\ \text{is embedded for}\ \tau\in[0,t)\right\}.
\]

Phisycally speaking, Theorem~\ref{TMain} says that a bubble may collapse in finite time which is often observed for air bubbles being in an incompressible fluid.

The main idea to prove Theorem~\ref{TMain} is to construct a solution after the boundary has a self-intersection.
 This strategy goes back to \cite{GI} where the existence of self-intersection for an evolving curve by surface diffusion.
 The main difference from \cite{GI} is that one has to construct a velocity field $u$ in $\Omega(t)$ when $\partial\Omega(t)$ is not embedded.
 Actually, this idea has been carried out for the free boundary problem \eqref{ENS} and \eqref{EIn} when there is no surface tension, i.e., $\sigma=0$ by D.~Coutand and S.~Shkoller \cite{CS}.
 What they proved is a result similar to Theorem~\ref{TMain} when $\sigma=0$ and $\Omega$ is bounded.
 Their result characterizes the phenomenon of the self-intersection for the boundary of a water droplet.
 In \cite{CS}, they used a Lagrange coordinate to reduce a problem to a problem in a fixed domain.
 It turns out this formulation is still valid to handle the evolution of $\Omega_\eta(t)$ after it has a self-intersection in $\mathbb{R}^d$ when $\sigma=0$. 

However, if $\sigma>0$, it seems that the formulation by the Lagrange coordinate does not work even for a standard local existence theory.
 Many researchers use the Hanzawa transform to reduce the free boundary problem to a problem in a fixed domain.
 The Hanzawa transform has been introduced by E.-I.~Hanzawa \cite{H} to construct a local-in-time classical solution to the Stefan problem which is a typical free boundary problem.
 However, when $\partial\Omega_\eta(t)$ has self-intersection and $\Omega_\eta(t)$ is no more embedded in $\mathbb{R}^d$, the original formulation of the Hanzawa transform is not available.
 To overcome this difficulty, we consider a domain with $\delta$-wing for a splash domain.
 This is locally a $\delta$-tubular neighborhood of the boundary but near self-intersection of the boundary, overlapped parts are considered different sets.
 This consideration can be realized mathematically by the notion of gluing manifolds.
 We give a precise definition in Section~\ref{SD}.
 Let $\Omega_\delta$ denote a domain with $\delta$-wing constructed from a splash domain $\Omega$.
 This set $\Omega_\delta$ can be regarded as a Riemannian manifold with flat metric but not embedded in $\mathbb{R}^d$.
 We shall work in $\Omega_\delta$ instead of  $\mathbb{R}^d$.
 Let $\hat{\Gamma}$ be a boundary of $\Omega$ in $\Omega_\delta$.
 For a given height function $h$ defined on $\hat{\Gamma}$ with $|h|\le\delta_1$ ($<\delta$) we set
\[
    \Omega^h:=\left\{x=y+s\mathbf{n}(y) \bigm|
    -\delta_1\le s <h(y),\ 
    y\in\hat{\Gamma}\right\}
    \cup \left\{x\in\Omega \bigm|
    \operatorname{dist}(x,\partial\Omega)>\delta_1\right\}
\]
for sufficiently small $\delta_1$ so that the normal coordinate is available for $|s|\le\delta_1$.
 Here, $\mathbf{n}$ denotes the exterior unit normal to $\hat{\Gamma}$ in $\Omega_\delta$.
 This set $\Omega^h$ is not embedded in $\mathbb{R}^d$ unless $h\le0$.
 The Hanzawa transform $\Xi_h$ we introduce is a $C^k$ ($k\ge2$) transform from $\Omega$ to $\Omega^h$ when $\hat{\Gamma}$ is $C^{k+1}$ and $h$ is a $C^k$ function. 
Near $\hat{\Gamma}$,
\[
    \Xi_h(y) = y+Lh(y),
\]
where $L$ is a lift of a function on $\hat{\Gamma}$ to $\Omega$.
 Away from a neighborhood of $\hat{\Gamma}$, we set $\Xi_h(y)=y$.
 If $\hat{\Gamma}$ is $C^2$ and $\|h\|_\infty$, $\|\nabla h\|_\infty$ are small, it turns out $\Xi_h$ gives a diffeomorphism.
 Here $\|\cdot\|_\infty$ denotes the $L^\infty$ norm on $\hat{\Gamma}$.
 More detailed properties as well as the definition will be given in Section~\ref{SH}.

We postulate that the moving domain $\Omega(t)$ in a free boundary problem \eqref{ENS} represented by $\Omega^{h(t)}$ with $h:\hat{\Gamma}\times[0,T)\to(-\delta_1,\delta_1)$ with initial data $\Omega_0=\Omega^{h(0)}$, where $h(t)=h(\cdot,t)$.
We transform $\Omega^{h(t)}$ to $\Omega$ by using the inverse Hanzawa transform $Z_{h(t)}$ given in this paper.
 Then the system \eqref{ENS}, \eqref{EIn} is reduced to a system of an equation for $(u,\varpi,h)$ with $h|_{t=0}=h_0$, $u|_{t=0}=u_0$.
 We have to solve this system locally-in-time with careful checking of the life span after $\partial\Omega(t)$ loses its embeddedness in $\mathbb{R}^d$.
 Fortunately, there is a very nice and thorough work by Y.~Shibata \cite{Sh} who proved unique existence of local-in-time solution for \eqref{ENS}--\eqref{EIn} when $\Omega$ is uniformly $C^3$-domain with estimate of life span.
 His result includes a bounded and an exterior domain as a special case.
 We notice that his result is still valid when $\Omega$ is a splash $C^3$-domain, because the problem is reduced to a local perturbed half-space problem and a whole space.
 We shall briefly discuss such a well-posedness result.

Once such a good well-posedness result is available, we are able to construct a sequence of initial data $h_{0,\eta}$ and $u_{0,\eta}$ so that the boundary $\partial\Omega_\eta(t)$ is forced to collapse in finite time by estimating the speed.
 This strategy is the same as in \cite{CS}.
 One has to be careful so that the existence time (life span) is estimated from below uniformly with respect to $\eta\downarrow0$.
 Also one has to be careful in constructing $u_{0,\eta}$ because it must satisfy the compatibility condition \eqref{EComp}.
 As in \cite{CS} we use a solvability result by \cite{SS71}.
 For an exterior domain, we further use the Bogovski's lemma (see e.g.\ \cite{Gal}) to solve $\operatorname{div}w=f$ to recover $\operatorname{div}u_{0,\eta}=0$ from cutting off procedure.
 (For more general domain having non-compact boundary, the construction of suitable $u_{0,\eta}$ should be more involved.)

There are several papers constructing a splash singularity other than \cite{CS}.
 In \cite{CCFGG}, a splash singularity is constructed for \eqref{ENS} and \eqref{EIn} with $\sigma=0$ in the plane $\mathbb{R}^2=\mathbb{C}$ by using the complex mapping $z\mapsto\sqrt{z}$ to transform ``a domain with overlapping" to a domain with no overlapping.
 This idea goes back to the work by \cite{CCFGG13}, where formation of splash singularity is proved for the one-phase water wave equation which has no viscosity.
 The idea of \cite{CS} is extended for the equations of magneto-hydrodynamics by \cite{HY} and equations for viscoelastic fluid models by \cite{DMS1}, \cite{DMS2} when $\sigma=0$.

Mathematical analysis on the free boundary problems for the Navier-Stokes equations for an initial domain with embedded smooth boundary  has been started by V.~A.~Solonnikov \cite{Sol77}, where he proved local-in-time existence for $\sigma>0$ when $\Omega_0$ is bounded.
 By now, it has a long history.
 The problem is roughly classified by a shape of initial domain $\Omega_0$ either a bounded domain or an ocean like domain.

For a bounded domain, local in time unique existence theorem was proved by V.~A.~Solonnikov \cite{Sol86}, \cite{Sol90}, \cite{Sol91}, \cite{Sol92} for $\sigma>0$.
 For $\sigma=0$, it was first established by \cite{Sol88}.
 His analysis did not appeal to energy estimate but maximal regularity estimates for a linearized equations.
 The reader is referred to a book \cite{DS21} a review article \cite{SD18} especially for activity of Russian school.

Even we restrict ourselves only on unique local existence results, there is a number of results depending on choice of function spaces.
 We do not mention them in this paper.
 We rather refer a recent book \cite{DS21} as well as nice lecture notes by Y.~Shibata \cite{Sh}.

If initial domain is ocean-like with finite depth, unique local existence theorem was first proved by J.~T.~Beale \cite{B80} for $\sigma=0$ under gravitational force for $d=3$ by using $L^2$-Sobolev spaces.
 Such a problem has been further studied by G.~Allain \cite{Al87} and A.~Tani \cite{Ta96} when $\sigma>0$ in the framework of $L^2$ Sobolev space and by H.~Abels \cite{A05} when $\sigma=0$ in $L^p$ framework.

More recently, Y.~Shibata \cite{Sh16}, Y.~Shibata and S.~Inna \cite{ShI} established the local well-posedness in general unbounded domain under the assumption that the initial domain is uniformly $C^3$ and the weak Dirichlet problem is uniquely solvable in itinial domain.
 In \cite{Sh16}, the case $\sigma=0$ was studied by using the Lagrange transform while in \cite{ShI}, the case $\sigma>0$ was studied by using the Hanzawa transform.
 The lecture note \cite{Sh} explains both methods.
 For an exterior domain, the local and global-in-time unique existence theorem were proved by Y.~Shibata \cite{Sh17}, \cite{ShCPAA}, \cite{Sh18} for $\sigma=0$.
 In his setting, $L^q$ in space direction while $L^p$ in time.
 One reason he uses such different exponents is that he would like to find some decay at time infinity to study large time behavior.
 Although we do not study any global-in-time solution, we essentially use Y.~Shibata's local-in-time existence theorem in \cite{Sh}.

Although this paper does not require any global-in-time results, we notice that this is an active area of current researches.
 Let us explain the case of an ocean like domain with finite depth.
 For $\sigma=0$, L.~Sylvester \cite{Sy} first proved global existence results for small data.
 Since then there has been significant progress for this problem either $\sigma>0$ and $\sigma=0$.
 For example, Y.~Guo and I.~Tice \cite{GT13a}, \cite{GT13b} studied in the case $\sigma=0$ and proved that solutions decay to equilibrium at an algebraic rate by introducing a new high-regularity energy method.
 The reader is referred to a recent paper by \cite{SaSh26} for more references.
 If the depth is infinite, the global existence problem is essentially more difficult because Poincar\'e inequalities do not work.
 If there is no surface tension, the problem is still parabolic if one uses the Lagrange coordinates.
 See \cite{OSh22}, \cite{DHMT25}, \cite{OgSh21}, \cite{ShW}.
 If there is surface tension, some hyperbolic effect appears in transport equation so the problem is more difficult.
 Fortunately, H.~Saito and Y.~Shibata \cite{SaSh24}, \cite{SaSh26} are able to overcome this difficulty and prove the global solvability result.

This paper is organized as follows.
 In Section~\ref{SD}, we introduce the notion of a domain with $\delta$-wing.
 In Section~\ref{SH}, we introduce the Hanzawa transform in our setting and derive a few important properties.
 In Section~\ref{SL}, we prove a more precise version of Theorem~\ref{TMain}.
 In Section~\ref{SW}, we explain the reason why the unique local-in-time existence of solution is adjustable for a splash domain by sketching the proof.

%%%%%%%%%%%%%%%%%%%%%%%%%%%%%%%%%%%%%%%%%%%%%%%%%
\section{Domain with $\delta$-wing} \label{SD} % Section 2

We consider a domain $\Omega$ in $\mathbb{R}^d$ whose boundary $\Gamma=\partial\Omega$ is not necessarily an embedded manifold.
 We then consider a flat Riemannian manifold which covers $\delta$-neiborhood of $\Omega$ but is not embedded in $\mathbb{R}^d$.
 This manifold is a key ambient space to study our free boundary problem.

Let us state these concepts in a precise way.
 Let $B_r^d(P)$ denote an open ball in $\mathbb{R}^d$ with raduis $r$ centered at $P\in\mathbb{R}^d$, i.e.,
\[
    B_r^d(P) = \left\{ x\in\mathbb{R}^d \bigm| |x-P|<r \right\}.
\]
It is convenient to consider a regular cylinder
\begin{align*}
    C_r^d(P) &= \left\{ x=(x',x_d)\in\mathbb{R}^d \bigm|
    |x'|<r,\ |x_d|<r \right\} \\
    &= B_r^{d-1}(0) \times (-r,r),
\end{align*}
where $0$ is the origin of $\mathbb{R}^{d-1}$.
 Let $f$ be a $C^k$ function defined on $\overline{B_r^{d-1}(0)}$, i.e.,
\[
    f\in C^k \left( \overline{B_r^{d-1}(0)} \right).
\]
We assume 
\begin{equation} \label{EAF1}
    f(0)=0, \quad \nabla'f(0)=0,
\end{equation}
where $\nabla'=\left(\partial/\partial x_1, \ldots, \partial/\partial x_{d-1}\right)$.
 We set
\[
    K_f :=\sup \left\{ \left| \partial_{x'}^\alpha f(x') \right| \Bigm|
    x'\in B_r(0),\ |\alpha|\le k \right\},
\]
where $\partial_{x'}^\alpha=\partial_{x_1}^{\alpha_1}\cdots\partial_{x_{d-1}}^{\alpha_{d-1}}$, $\partial_x=\partial/\partial x_i$, $|\alpha|=\alpha_1+\cdots +\alpha_{d-1}$, $\alpha_i\in\mathbb{N}\cup\{0\}$.
 To simplify the notation, we often drop the superscript $d$ and $d-1$ of $B_r$ and $C_r$.
 Since $f\in C^k\left(\overline{B_r(0)}\right)$, the constant $K_f$ is finite.

We assume that
\begin{equation} \label{EAF2}
    \delta_f := \min \left\{ r+\inf f,\ r-\sup f \right\}>0,
\end{equation}
where $\sup$ and $\inf$ are taken on $B_r(0)$.
 Let us denote
\[
    C_{r,f} = \left\{ x\in C_r \bigm| x_d>f(x') \right\}, 
\]
which is an epi-graph of $f$.
\begin{definition} \label{DIm}
Let $\Omega \subset \mathbb{R}^d$ be a domain and $r, \delta_0, K > 0$.
We say that $\Omega$ is a uniformly $C^k$ ($k\ge1$) \emph{splash domain} with boundary $\Gamma=\partial\Omega$ of type $(r,\delta_0,K)$ if for each point $P \in \Gamma$,
\begin{enumerate}
    \item[(i)] $\Omega \cap B_{2r}(P)$ has finitely many connected components $D_1(P), \ldots, D_{m(P)}(P)$ with $m(P)$ denoting the number of connected components in $\Omega\cap B_{2r}(P)$;
    \item[(i\hspace{-0.1em}i)] for every $1 \leq j \leq m(P)$, there exist a unique rotation $Q_{j,P}$ and a unique scalar function $f_{j,P}\in C^k\left(\overline{B_r(0)}\right)$ satisfying \eqref{EAF1}, \eqref{EAF2} with $K_{f_{j,P}}\le K$, $\delta_0\le\delta_{f_{j,P}}$ such that
    \[
        D_j(P) \cap (Q_{j,P} C_r+P)
        = Q_{j,P} C_{r,f_{j,P}} +P
    \]
    provided that $P\in\partial D_j(P)$;
    \item[(i\hspace{-0.1em}i\hspace{-0.1em}i)] there exists a point $P_0 \in \Gamma$ such that $P_0 \in \partial D_j(P_0)$ for two $j$'s.
\end{enumerate}
If $(r, \delta_0, K)$ is allowed to depend on $P$, we simply say that $\Omega$ is a $C^k$ splash domain.
\end{definition}
\begin{rem} \label{RIm}
    \begin{enumerate}
        \item[(i)] By geometry, there are at most two $j$'s satisfying $P \in \partial D_j(P)$ since $k \ge 1$; see Figure~\ref{FD}.
 In addition, we note that it is possible to have $1 \leq j \leq m(P)$ such that $P \notin \partial D_j(P)$; see $D_3(P)$ in Figure~\ref{FD}.
        \begin{figure}[H]
            \centering
            \includegraphics[width=0.4\linewidth]{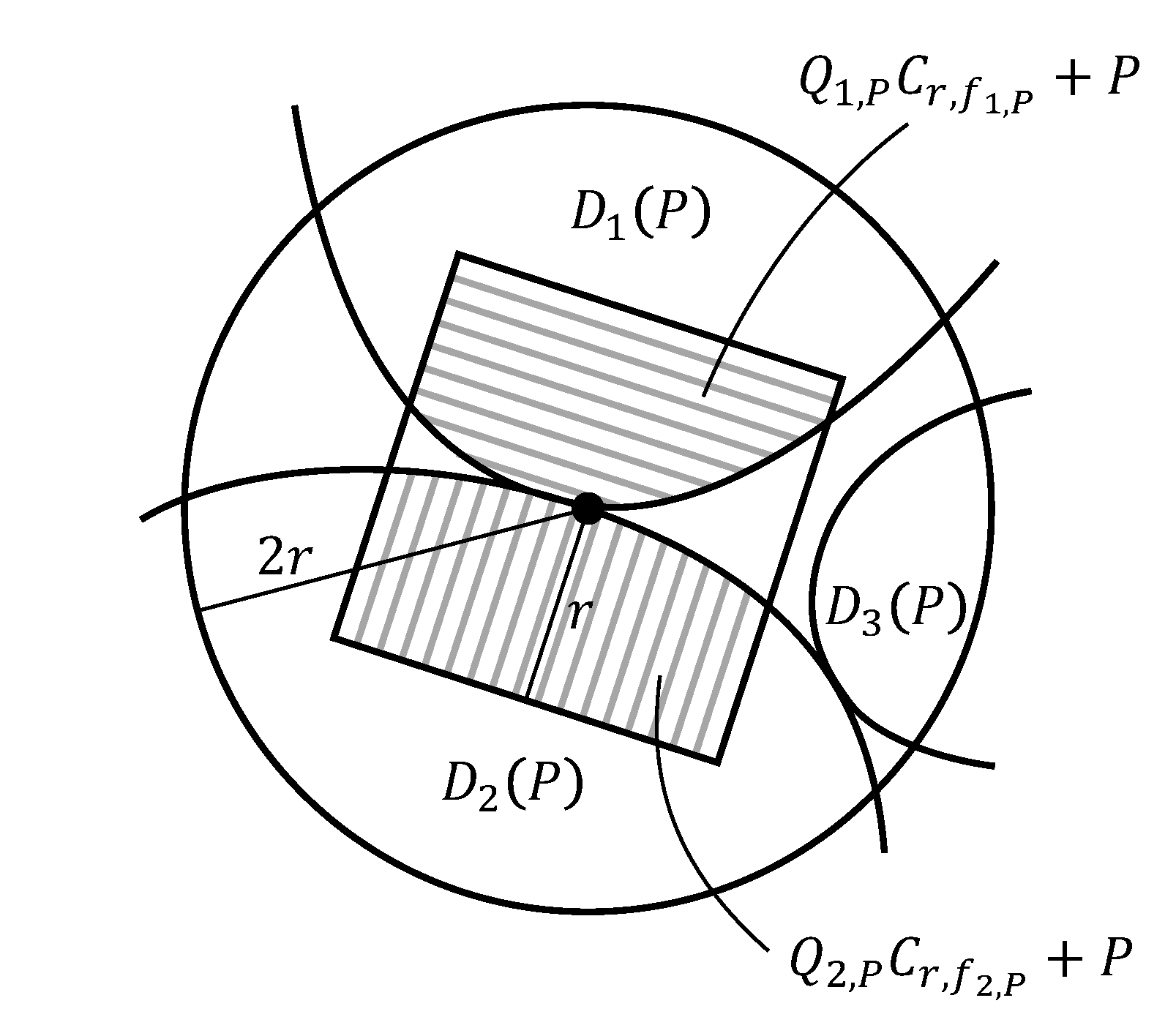}
            \caption{Connected components}
            \label{FD}
        \end{figure}
        If there are two $j$'s satisfying $P \in \partial D_j(P)$, we say that $P$ is a self-intersection point.
        If $P$ is a self-intersection point and $j=1,2$, then $Q_{2,P}=RQ_{1,P}$ where $R$ is a $180^\circ$-degree rotation such that $R\mathbf{n}=-\mathbf{n}$ where $\mathbf{n}$ is the unit normal at $P$ on $\partial D_1$.
        Moreover, $Q_{2,P}C_{r,f_{2,P}}=Q_{1,P}(C_r\setminus\bar{C}_{r,\tilde{f}_{2,P}})$ with some $\tilde{f}_{2,P}$.
        \item[(i\hspace{-0.1em}i)] In Definition~\ref{DIm}, if condition (i\hspace{-0.1em}i\hspace{-0.1em}i) does not occur, i.e., $m(P) = 1$ for all $P\in\Gamma$, then $\Omega$ is called a uniformly $C^k$ domain in the literature; see e.g.\ \cite{BG}.
        \item[(i\hspace{-0.1em}i\hspace{-0.1em}i)] Assume that the boundary is $(r,\delta_0,K)$-type.
         Then it is $(\rho,\delta_{0\rho},K)$-type for any $\rho<r$ with $\delta_{0\rho}=\delta_0-(r-\rho)$ provided that $\delta_{0\rho}>0$.
         Moreover, if $\rho$ is sufficiently small, we may take $\delta_{0\rho}/\rho < 1$ close to $1$ because of the continuity of normals of $\Gamma$.
         Furthermore, if $\rho\to0$, then $K\to0$.
         \item[(i\hspace{-0.1em}v)] If $\Gamma$ is compact, a $C^k$ ($k\ge1$) splash domain is a uniformly $C^k$ splash domain with boundary of type $(r,\delta_0,K)$ for some $r_0>0$, $\delta_0\in(0,r_0)$, $K<\infty$.
    \end{enumerate}
\end{rem}

We next construct the $\delta$-wing of $\Omega$.
 We set
\[
    C_{r,f,\delta} = \left\{ x\in C_r \bigm|
    d(x,C_{r,f})<\delta \right\}
    \quad\text{with}\quad \delta<\delta_f.
\]
\begin{definition} \label{DW}
    Let $\Omega$ be a uniformly $C^k$ ($k\ge1$) splash domain with boundary of type $(r,\delta_0,K)$.
     For $\delta<\delta_0$, we set
    \[
       \hat{U_\delta} = \coprod \left\{ \left(Q_{j,P} C_{r,f_{j,P},\delta}+P \right) \setminus D_j(P)
        \bigm| 1\le j\le m(P),\ P\in\partial\Omega,\ P\in\partial D_j(P)\right\}
    \]
    be a disjoint union.
     If
    \[
       (Q_{j,P} C_{r,f_{j,P}}+P) \cap
       (Q_{j',P'} C_{r,f_{j',P'}}+P') \neq \emptyset, \quad\text{for}\quad
       P, P'\in\partial\Omega, \quad
       1\le j\le m(P),\ 1\le j'\le m(P'),
    \]
    then for
    \[
       x\in(Q_{j,P} C_{r,f_{j,P},\delta}+P) \setminus D_j(P), \quad
       y\in(Q_{j',P'} C_{r,f_{j',P'},\delta}+P') \setminus D_{j'}(P'),
    \]
    we say $x\sim y$ if $x=y$ as a point of $\mathbb{R}^d$.
    This is an equivalence relation and the $\delta$-wing of $\Omega$ can be defined by the equivalence class
    \[
       U_\delta=\hat{U}_\delta/\sim;
    \]
    see Figure~\ref{FI} and Figure~\ref{FN}.
\begin{figure}[htb]
\centering
\begin{minipage}[b]{0.48\columnwidth}
    \centering
            \includegraphics[width=1\linewidth]{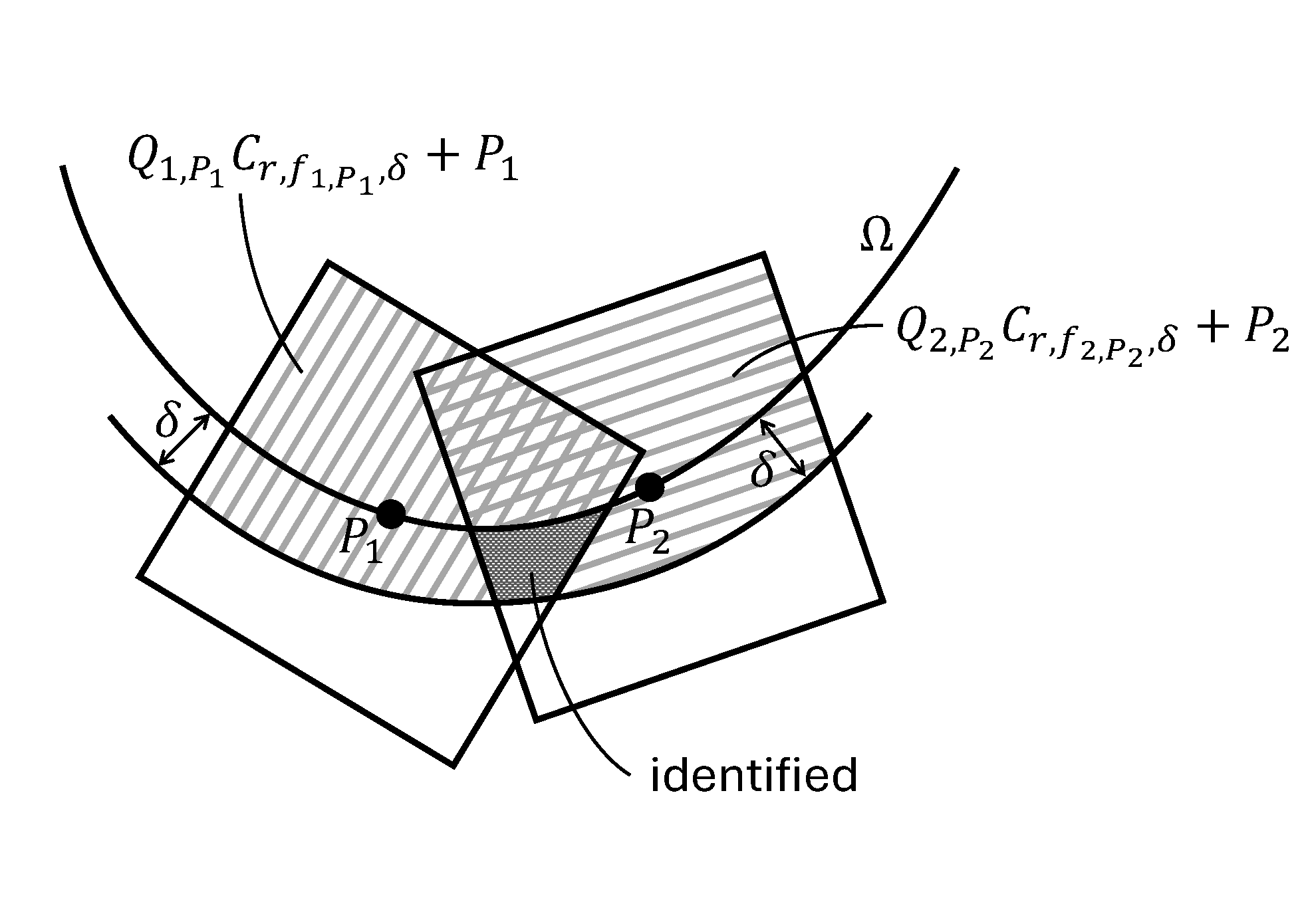}
            \caption{Identified overlap}
            \label{FI}
\end{minipage}
\begin{minipage}[b]{0.4\columnwidth}
    \centering
            \includegraphics[width=1\linewidth]{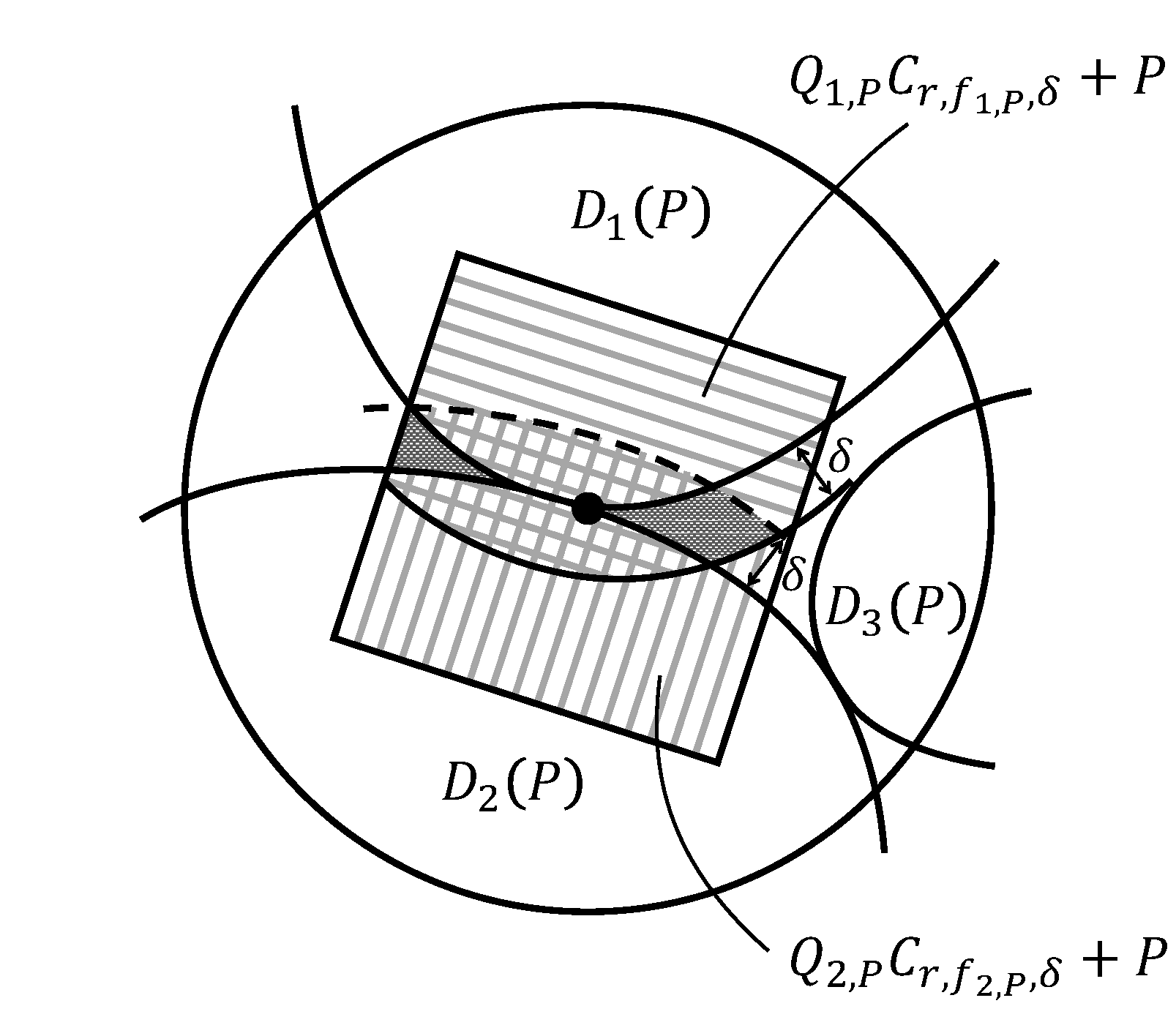}
            \caption{Non-identified overlap}
            \label{FN}
\end{minipage}
\end{figure}
\end{definition}
\begin{rem} \label{RN}
    If we define the $\delta$-wing $U_\delta$ for a uniformly $C^k$ domain $\Omega \subset \mathbb{R}^d$ following Definition~\ref{DW}, then this $U_\delta$ would agree with
    \[
        V_\delta=\delta\text{-neighborhood of}\ \Bar{\Omega} \setminus\Omega.
    \]
    Let $\pi$ be a canonical projection from $U_\delta$ to $\mathbb{R}^d$, i.e., $\pi(x)$ for $x\in U_\delta$ is regarded as a point in $\mathbb{R}^d$, then $\pi(U_\delta)=V_\delta$ even if $\Omega$ has a non-embedded boundary.
\end{rem}
\begin{definition} \label{DDW}
Let $U_\delta$ be the $\delta$-wing of $\Omega$.
 We set
\[
    \Omega_\delta=U_\delta \coprod\Omega \quad\text{(disjoint union)}
\]
and call it a \emph{domain with $\delta$-wing} constructed from $\Omega$.
 The set $\Omega_\delta$ can be regarded as a Riemannian manifold with flat metric by introducing a coordinate system as follows.
 Let $\iota_{j,P}$ be the mapping from $C_{r, f_{j,P}, \delta}$ to $\Omega_\delta$ defined by
\begin{equation*}
\iota_{j,P} (x) =
\begin{cases}
[Q_P x+P], & x\in C_{r,f_{j,P},\delta} \setminus C_{r,f_{j,P}}, \\
Q_P x+P, & x\in C_{r,f_{j,P}},
\end{cases}
\end{equation*}
where $[\quad]$ denotes the equivalence class in $U_\delta$.
 The coordinate system for $\Omega_\delta$ is given by 
\begin{align*} 
\{ \iota_{j,P} \}_{P \in \partial \Omega, \, 1 \leq j \leq m(P)} \cup \{ \iota \},
\end{align*}
 where $\iota:\Omega\to\Omega$ is the identity.
 (Although the manifold $\Omega_\delta$ is not embedded in $\mathbb{R}^d$, the canonical projection $\pi(\Omega_\delta)$ is nothing but the $\delta$-neighborhood of $\Omega$ in $\mathbb{R}^d$.)
\end{definition}

We shall fix coordinate patches of $\Omega_\delta$.
For $\delta < \delta_0 - r/2$, we consider the manifold $\Omega_\delta$ with $\delta$-wing defined by Definition~\ref{DDW}.
 Let $d_{\Omega_\delta}(x,y)$ be the geodesic distance of two points in $\Omega_\delta$, i.e., it is the minimal length of a curve connecting $x$ and $y$ in a manifold.
 We use $\hat{\Gamma}$ to denote the boundary of $\Omega$ in the manifold $\Omega_\delta$, i.e., $\hat{\Gamma}$ is an immersed $d-1$ dimensional manifold in $\mathbb{R}^d$. We note that the canonical projection $\pi:\hat{\Gamma}\to\Gamma$ may not be injective.
It turns out that for sufficiently small $\delta_1\le\delta$,
\begin{equation} \label{EU}
    U_{\delta_1} = \left\{ x\in U_\delta \bigm|
    d_{\Omega_\delta}(x,\hat{\Gamma})<\delta_1 \right\}.
\end{equation}
Indeed, since $\delta_1<\delta_0-r/2$, in each $C_{r, f_{j,P},\delta_1}$ the boundary effect near $|x'| = r$ can be neglected in $C_{r/2}$, i.e.,
\[
    \tilde{U}_{\delta_1} \cap \iota_{j,P}(C_{r/2})
    = \iota_{j,P} \left( C_{r/2} \cap (C_{r,f_{j,P},\delta_1} \setminus C_{r,f_{j,P}}) \right)
\]
for any $P\in\Gamma$ and $f_{j,P}$, where $\tilde{U}_{\delta_1}$ is the right-hand side of \eqref{EU}.

Let $\Omega$ be a $C^k$ ($k\ge3$) splash domain with compact boundary $\Gamma$.
 By Remark~\ref{RIm} (i\hspace{-0.1em}i), $\Omega$ is a uniformly $C^k$ splash domain with boundary of type $(r,\delta_0,K)$ for some $r,\delta_0,K > 0$.
 By Remark~\ref{RIm} (i\hspace{-0.1em}i\hspace{-0.1em}i) we may assume $\delta_0>r/2$.
 Since $\Gamma$ is compact, there exists finitely many points $P_i\in\Gamma$ ($1\le i\le M$) such that
\begin{equation} \label{ELS}
    \mathcal{U}_r= \left\{ \iota_{j,P_i} (C_{r/2} \cap C_{r,f_{j,P_i},\delta})
    \bigm| 1\le i \le M, \ 
    1\le j\le m(P_i)\ \text{with}\ P_i\in\partial D_j(P_i) \right\}
\end{equation}
covers $\delta$-neighborhood of $\hat{\Gamma}$ in $\Omega_\delta$.
 This can be regarded as a local coordinate patch near $\hat{\Gamma}$.

As usual, we introduce a normal coordinate system.
 For this purpose, we assume that $k\ge2$ so that $\hat{\Gamma}$ is a $C^k$ ($k\ge2$) immersed manifold in $\mathbb{R}^d$.
 In this setting, it is well known that for each $x\in\Omega_{\delta_1}$, there is a unique point $\Pi(x)\in\hat{\Gamma}$ such that
\[
    d_{\Omega_\delta}(x,\hat{\Gamma})
    = \left| x-\Pi(x) \right|
\]
provided that $\delta_1$ is chosen sufficiently small (cf.\ \cite[14.6]{GT}).
 (In other words, $\hat{\Gamma}$ has a positive reach.)
 Let $\mathbf{n}(x)$ be an exterior normal of $\hat{\Gamma}$ of $\Omega$ in $\Omega_{\delta_1}$.
 The normal coordinate system is of the form
\begin{equation} \label{ECor}
    x = y+ s \mathbf{n}(y), \quad
    y\in\hat{\Gamma},
\end{equation}
where $y=\Pi(x)$ and $s \in \mathbb{R}$ is defined by
\begin{equation*}
	s =
     \begin{cases}
     \left| x-\Pi(x)\right|, & x\in U_\delta, \\
     -\left| x-\Pi(x)\right|, & x\in\Omega.
     \end{cases}
\end{equation*}
This function is the signed distance from $x$ to $\hat{\Gamma}$ in $\Omega_\delta$.
 This relation \eqref{ECor} is a $C^{k-1}$ coordinate system near $\hat{\Gamma}$.

%%%%%%%%%%%%%%%%%%%%%%%%%%%%%%%%%%%%%%%%%%%%%%%%%

\section{Hanzawa transform} \label{SH} % Section 3

In this section, we introduce a Hanzawa transform following \cite{Sh}, \cite[Chapter~1, \S3.2 (a)]{PS}.
From now on, to simplify the argument we consider a domain $\Omega$ in $\mathbb{R}^d$ with compact boundary.
In other words, $\Omega$ is either a bounded or an exterior domain.
Assume that $\Omega$ is a $C^k$ ($k\ge2$) splash domain.
Let $\Omega_\delta$ be a domain with $\delta$-wing constructed from $\Omega$.
Let $\hat{\Gamma}$ denote its boundary in $\Omega_\delta$ which is a $d-1$ dimensional $C^k$ manifold but it is not embedded in $\mathbb{R}^d$.
We shall recall a few known results.

%%%%%%%%%%%%%%%%%%%%%%%%%%%%%%%%%%%%%%%%%%%%%%%%%
\subsection{Trace and lifting} \label{SSS} % Subsection 3.1

Let $W^{m-1/q,q}(\hat{\Gamma})$ be the Sobolev space on $\hat{\Gamma}$ for $m\in\mathbb{N}$.
 It can be defined the space of trace from $W^{m,q}(\Omega)$.
 In fact, let $\mathcal{U}_r$ be a coordinate system of the $\delta$-neighborhood of $\hat{\Gamma}$ in $\Omega_\delta$ defined by \eqref{ELS}.
 We may assume that $\mathcal{U}_{r'}$ for $r'<r$ close to $r$ already covers $\delta$-neighborhood of $\hat{\Gamma}$ in $\Omega_\delta$.
 Let $\{\varphi_{ij}\}$ be a partition of the unity associated with the coordinate system.
 In other words, $\varphi_{ij}\in C_c^\infty\left(\iota_{j,P_i}(C_{r/2})\right)$ satisfies
\begin{gather*}
	\varphi_{ij}>0 \ \text{on}\  \iota_{j,P_i}(C_{r'/2}), \quad
	(i,j)\in\Lambda, \\
	\operatorname{spt}\varphi_{ij} \subset \iota_{j,P_i}(C_{r/2}), \quad
	(i,j)\in\Lambda, \\
	\sum_{(i,j)\in\Lambda} \varphi_{ij} \equiv 1
	\ \text{on}\ \delta\text{-neighborhood of}\ \hat{\Gamma}\ \text{in}\ \Omega_\delta,
\end{gather*}
where $\Lambda=\left\{ (i,j)\bigm| 1\le i\le M,\ 1\le j\le m(P_i),\ P_i\in\partial D_j(P_i) \right\}$.
 (For a function supported in $\iota_{j,P_i}(C_{r/2})$, we regard it a function on $\Omega_\delta$ by a zero extension to $\Omega_\delta\setminus\iota_{j,P_i}(C_{r/2})$.)
 For $w\in W^{m,q}(\Omega)$, the trace $\gamma w$ on $\hat{\Gamma}$ (not $\partial\Omega$) is defined as
\[
	\gamma w = \sum_{(i,j) \in \Lambda} \gamma(\varphi_{ij}w),
\]
where the trace of $\varphi_{ij}w$ is taken in
\[
	\iota_{j,P_i} (C_{r/2} \cap C_{r,f_{j,P_i,\delta}}).
\]
In other words, we take trace of 
\[
	i_{j,P_i}^\#(\varphi_{ij}w)(x)
	:= (\varphi_{ij}w) \left(i_{j,P_i}(x)\right)
\]
to $x_d=f_{j,P_i}(x')$ and regard as a function on $\hat{\Gamma}\cap\iota_{j,P_i}(C_{r/2})$.
We define
\begin{align} \label{trDef:frSob}
W^{m-1/q}(\hat{\Gamma}) := \gamma W^{m,q}(\Omega).
\end{align}
This is well defined since it is independent of the choice of coordinate system $\mathcal{U}_r$.
 On the other hand, the Besov space $B_{q,p}^s$ ($s \in \mathbb{R}$, $1 \leq p,q \leq \infty$) on $\hat{\Gamma}$ can be defined as follows.
 For a scalar function $h$ defined on $\hat{\Gamma}$, we define $h \in B_{q,p}^s(\hat{\Gamma})$ if and only if 
\begin{align*}
(\varphi_{ij} h) \circ \iota_{j, P_i} \in B_{q,p}^s(\mathbb{R}^{d-1}), \quad \forall \; (i,j) \in \Lambda
\end{align*}
and
\begin{align*}
\| h \|_{B_{q,p}^s(\hat{\Gamma})} := \sum_{(i,j) \in \Lambda} \| (\varphi_{ij} h) \circ \iota_{j, P_i} \|_{B_{q,p}^s(\mathbb{R}^{d-1})} < \infty.
\end{align*}
When $s \notin \mathbb{Z}$, the Besov space $B_{q,q}^s(\hat{\Gamma})$ agrees with the Sobolev-Slobodeckij space $W^{s,q}(\hat{\Gamma})$. 
In particular, $B_{q,q}^{m - 1/q}(\hat{\Gamma})$ is an equivalent definition for the trace space \eqref{trDef:frSob}.

To define a Hanzawa transform we need a lift from $W^{m-1/q}(\hat{\Gamma})$ to $W^{m,q}(\Omega)$ for $q\in(1,\infty)$.
 Here is a typical way to lift a function defined in a $d-1$ dimensional unit cube $C=\left\{|x'|<1\right\}$ to a function defined in a pyramid of the form
\[
	P_d = \left\{ x\in\mathbb{R}^d \bigm|
	0<x_d<1,\ |x_i|<1-x_d,\ i=1,\ldots,d-1 \right\}.
\]
It turns out that the linear operator
\[
	L(g)(x',x_d) = \frac{1}{x_d^{d-1}} \int_{\mathbb{R}^{d-1}}
	R \left( \frac{y'-x'}{x_d} \right) g(y')\, dy'
\]
is continuous from $W^{m-1/q}(C)$ to $W^{m,q}(P_d)$, where $R\in C_c^\infty(\mathbb{R}^{d-1})$ satisfying $\int_{\mathbb{R}^d}R\,dx=1$; see \cite[Lemma~5.6 in Section~2.5.5]{Ne}.
 For the case of a bounded Lipschitz domain $\Omega$ with boundary $\Gamma$, a continuous linear operator (bounded linear operator) $L:W^{m-1/q,q}(\Gamma)\to W^{m,q}(\Omega)$ can then be constructed using techniques of localization by partition of unity and flattening boundaries by a suitable coordinate changes; see \cite[Theorem~5.8 in Section~2.5.7]{Ne}.
 For $g\in W^{m-1/q,q}(\Gamma)$, the function $L(g)$ is called a continuous lifting of $g$ (by $L$).
 Let us go back to our setting.
 Let $g$ be a function defined on
\[
	\Sigma_{ij} = \left\{ x\in C_r \bigm|
	f_{j,P_i}(x') = x^d \right\}
\]
supported on $C_{r'/2}\cap\Sigma_{ij}$.
 Since $C_{r,f_j,P_i}$ is a bounded Lipschitz domain, there is a continuous lifting operator $L_0$ from $W^{m-1/q,q}(\partial C_{r,f,P_i})$ to $W^{m,q}(C_{r,f_j,P_i})$.
 Since $g$ is supported on $C_{r'/2}\cap\Sigma_{ij}$, we take a cut off function $\varphi\in C_c^\infty(C_r)$ such that $\varphi\equiv1$ on $C_{r'/2}$ and $\operatorname{spt}\varphi\subset C_{r/2}$ and set $L_{ij}=\varphi L_0$.
\begin{lem} \label{LLi}
Let $\Omega$ be a $C^1$ splash domain with compact boundary $\hat{\Gamma}$.
 Then there is a continuous linear operator $L$ from $W^{m-1/q,q}(\hat{\Gamma})$ to $W^{m,q}(\Omega)$ for any $m\in\mathbb{N}$ such that $\gamma\left(L(g)\right)=g$.
 More precisely, there exists $C_L$ depending only on $C^1$-regularity such that
\[
    \left\|L(g)\right\|_{W^{m,q}(\Omega)} \le C_L\|g\|_{W^{m-1/q,q}(\hat{\Gamma})}
\]
for all $g\in W^{m-1,q}(\hat{\Gamma})$.
\end{lem}
\begin{proof}
For a given $g\in W^{m-1/q.q}(\hat{\Gamma})$ we decompose
\[
	g=\sum_{(i,j)\in\Lambda} \varphi_{ij} g
\]
and define
\[
	L(g)= \sum_{(i,j)\in\Lambda} \iota_{\#j,P_i} L_{ij} \left(\iota_{j,P_i}^\#(\varphi_{ij}g) \right),
\]
where $(\iota_{\#j,P_i}\psi)(x)=\psi(\iota_{j,P_i}^{-1}x)$ for a function $\psi$ defined on $C_{r/2,j_j,P}$.
 Evidently, $L$ is continuous from $W^{m-1/q.q}(\hat{\Gamma})$ to $W^{m,q}(\Omega)$.
 By definition, it is easy to see that $\gamma\left(L(g)\right)=g$.
\end{proof}
We often call $L$ a lifting operator from $W^{m-1/q,q}(\hat{\Gamma})$ to $W^{m,q}(\Omega)$.

%%%%%%%%%%%%%%%%%%%%%%%%%%%%%%%%%%%%%%%%%%%%%%%%%
\subsection{Hanzawa transform based on a Sobolev function} \label{SSH} % Subsection 3.2

For a sufficiently small $\delta_1$, $\Omega_{\delta_1}$ has a normal coordinate system near $\hat{\Gamma}$, i.e.,
\[
	x = y + s \mathbf{n}(y), \quad
	y = \Pi(x), \quad y \in \hat{\Gamma}, \quad
	s \in \mathbb{R}
\]
so that the mapping $(y, s)\mapsto x$ is a $C^{k-1}$ ($k \ge 2$) diffeomorphism from
\[
    \left\{ (y, s) \bigm|
    y \in \hat{\Gamma},\ 
    | s | \le \delta_1 \right\}
\]
to a tubular neighborhood
\[
    \left\{x\in\Omega_\delta \bigm|
    d_{\Omega_\delta}(x,\hat{\Gamma})\le \delta_1 \right\}
\]
of $\hat{\Gamma}$.
 For a given $C^\ell$ ($1\le\ell\le k-1$) function $h$ on $\hat{\Gamma}$, we define
\begin{equation} \label{ENo}
    \Omega^h = \left\{ x = y + s \mathbf{n}(y)\in\Omega_{\delta_1} \bigm|
	y\in\hat{\Gamma},\ - \delta_1 \leq  s <h(y) \right\} \cup
	\left\{ x\in\Omega \bigm| d (x,\partial\Omega)>\delta_1 \right\},
\end{equation}
where $\|h\|_\infty\le\delta_1$ and $d(x,\partial\Omega)$ denotes the distance of $x$ from $\partial\Omega$.
 This is a $C^\ell$ domain in $\Omega_{\delta_1}$ but not included in $\mathbb{R}^d$ unless $h\le0$ on $\hat{\Gamma}$.

We shall introduce a diffeomorphism called the Hanzawa transform which maps $\Omega$ to $\Omega^h$.
 Instead of $h\in C^\ell$ we consider a Sobolev function.
 For $h\in W^{m-1/q,q}(\hat{\Gamma})$, $m\in\mathbb{N}$ we see that $L(h)\in W^{m,q}(\Omega)$ by Lemma~\ref{LLi}.
 For $q>d$, then by the Sobolev embedding $Lh=L(h)\in C^{m-1}(\Omega\cup\hat{\Gamma})$ so that $h\in C^{m-1}(\hat{\Gamma})$.
 We take $m\ge2$ so that $\nabla_{\hat{\Gamma}}h$ is at least continuous.
 For $y\in\Omega$, we set
\begin{equation} \label{EHan}
	\Xi_h(y) := y+\xi_h(y), \quad
    \xi_h(y):=\chi\left(\frac{d(y,\partial\Omega)}{\delta_1} \right)(Lh)(y)\mathbf{n} \left(\Pi(y)\right)
\end{equation}
where $\chi\in C_c^\infty(\mathbb{R})$ is a cut off function such that $\chi(s)=1$ for $|s|\le1/2$ $\chi(s)=0$ for $|s|\ge1$ and $0\le\chi\le1$.
 Note that if $d(y,\partial\Omega)>\delta_1$ then $\Xi_h(y)=y$.
 By definition, we have the following regularity result.
\begin{prop} \label{PReH}
Assume that $\hat{\Gamma}$ is $C^k$ ($k\ge2$) and compact.
 Then there exists a constant $C$ depending only on $(r,\delta_0,K)$, $q \in(1,\infty)$ and $k$ such that
\[
    \|\xi_h\|_{W^{k,q}(\Omega)} \le C\|h\|_{W^{k-1/q,q}(\hat{\Gamma})}.
\]
\end{prop}
\begin{proof}
This is clear if we admit $\nabla^kd\in L^\infty(\hat{\Gamma})$ and
\[
    \|Lh\|_{W^{k,q}(\Omega)} \le C_L\|h\|_{W^{k-1/q,q}(\hat{\Gamma})}.
\]
\end{proof}

In the next subsection, we shall see that $\Xi_h$ is a diffeomorphism from $\Omega$ to $\Omega^h$ (up to the boundary) provided that $\|\nabla h\|_\infty$ is sufficiently small.
 This coordinate change is called the \emph{Hanzawa transform}; see e.g.\ \cite[Chapter~1, \S3.2~(a)]{PS}.

%%%%%%%%%%%%%%%%%%%%%%%%%%%%%%%%%%%%%%%%%%%%%%%%%
\subsection{Invertibility of the Hanzawa transform} \label{SSIH} % Subsection 3.3

We shall give a sufficient condition for $h$ so that the Hanzawa transform $\Xi_h$ is a diffeomorphism from $\Omega$ to $\Omega^h$ in the domain $\Omega_\delta$ with $\delta$-wing.

We begin with calculating the Jacobi matrix $\nabla\Xi_h$.
 We first notice that $\nabla d(y)=\mathbf{n}\left(\Pi(y)\right)$ for $y\in\Omega$ with $d(y)=d(y,\partial\Omega)<\delta_1$, where $\Pi(y)\in\hat{\Gamma}$.
 We use the convention that $(\nabla\Xi_h)_{ij}=\partial_{y_j}\Xi_h^i$.
 Then the Jacobi matrix $\nabla\Xi_h$ is of the form
\[
    \nabla\Xi_h = I+\nabla\xi_h, \quad
    \nabla\xi_h = N_h+D_h+R_h
\]
with
\begin{align*}
    &N_h(y) := \frac{1}{\delta_1} \chi' \left(\frac{d(y)}{\delta_1}\right)
    \mathbf{n}\left(\Pi(y)\right) \otimes \mathbf{n}\left(\Pi(y)\right)(Lh)(y), \\
    &D_h(y) := \chi \left(\frac{d(y)}{\delta_1}\right)
    \mathbf{n}\left(\Pi(y)\right) \otimes \nabla(Lh)(y), \\
    &R_h(y) := \chi \left(\frac{d(y)}{\delta_1}\right)
    \left(\nabla\mathbf{n}\left(\Pi(y)\right)\right)Lh(y),
\end{align*}
where $I$ denotes the identity matrix.
 The term $N_h$ is supported away from $\Gamma=\partial\Omega$ but near $\Gamma$.
 The term $D_h$ depends on derivative of $h$ and $R_h$ depends on the second fundamental form of $\hat{\Gamma}$.
\begin{lem} \label{LInv}
    The inverse matrix $(\nabla\Xi_h)^{-1}$ exists and it is of the form
\[
    (\nabla\Xi_h)^{-1}= \left(I-\frac{N_h+D_h}{K_h}\right)(I+R_h)^{-1}
\]
with a scalar function
\[
    K_h(y):= 1+\frac{1}{\delta_1}\chi' \left(\frac{d(y)}{\delta_1}\right) Lh(y)
    + \chi \left(\frac{d(y)}{\delta_1}\right) \nabla(Lh)(y)
    \cdot\mathbf{n}\left(\Pi(y)\right)
\]
provided that $K_h\neq0$ and $(I+R_h)^{-1}$ exists.
\end{lem}
\begin{proof}
We note $N_h+D_h$ is a rank one matrix.
 More precisely,
\[
    N_h+D_h=a\otimes b
\]
with $a=\mathbf{n}\left(\Pi(y)\right)$ and
\[
    b=\frac{1}{\delta_1}\chi' \left(\frac{d}{\delta_1}\right)
    \mathbf{n}\left(\Pi(y)\right)(Lh)(y)
    +\chi\left(\frac{d}{\delta_1}\right)\nabla(Lh)(y).
\]
We recall an identity from linear algebra
\[
    (I+a\otimes b)^{-1}=I-\frac{a\otimes b}{1+a\cdot b}
    \quad\text{unless}\quad a\cdot b=-1;
\]
see e.g.\ \cite[Chapter~2, \S2.1]{PS} and observe that
\begin{equation} \label{ELin}
    (I+N_h+D_h)^{-1}
    = I-\frac{N_h+D_h}{K_h}
\end{equation}
since $a\cdot b+1=K_h$.
 Since $\nabla \mathbf{n}\left(\Pi(y)\right)=\nabla^2d(y)$ is symmetric and $\mathbf{n}\left(\Pi(y)\right)\cdot\nabla_y\mathbf{n}\left(\Pi(y)\right)=0$, we see that
\[
    R_h(N_h+D_h)=0.
\]
We now apply \eqref{ELin} to get
\begin{align*}
    (\nabla\Xi)^{-1}
    =(I+N_h+D_h+R_h)^{-1}
    &=(I+N_h+D_h)^{-1}
    \left(I+R_h(I+N_h+D_h)^{-1}\right)^{-1} \\
    &=\left\{I-\frac{N_h+D_h}{K_h}\right\}(I+R_h)^{-1}.
\end{align*}
The proof is now complete.
\end{proof}

We now discuss a sufficient condition of $h$ such that $\Xi_h$ is a diffeomorphism from $\Omega$ to $\Omega^h$.
 As in \cite[\S3.1]{Sh}, if $\|\nabla\xi_h\|_{L^\infty(\Omega)}\le\varepsilon_0$ with $\varepsilon_0<1$, then $\Xi_h$ is injective in $\Omega$.
 Indeed, for $x_i=y_i+\xi_h(y_i)$ $i=1,2$
\begin{align*}
    |x_1-x_2|
    =\left|y_1+\xi_h(y_1)-\left(y_2+\xi_h(y_2)\right)\right|
    &\ge |y_1-y_2|-\left|\xi_h(y_1)-\xi_h(y_2)\right| \\
    &\ge (1-\varepsilon_0)|y_1-y_2|,
\end{align*}
which implies the injectivity.
 The condition $\|\nabla\xi_h\|_\infty\le\varepsilon_0$ is fulfilled if
\[
    \|Lh\|_{L^\infty(\Omega)} \le \varepsilon_0.
\]
We write the inverse map $\Xi_h^{-1}(x)=y$ as
\[
    \Xi_h^{-1}(x) = x+\Phi_h(x).
\]
\begin{thm} \label{TInv}
Let $\Omega$ be a $C^2$ splash domain with compact boundary of type $(r,\delta_0,K)$ and let $\Omega_\delta$ be a domain with $\delta$-wing constructed from $\Omega$.
 Let $\hat{\Gamma}$ be the boundary of $\Omega$ in $\Omega_\delta$.
 Let $\delta_1$($<\delta_0$) be taken such that normal coordinate is available in $\delta_1$ neighborhood of $\hat{\Gamma}$ in $\Omega_\delta$.
 Assume that $Lh|_{\hat{\Gamma}}=h$.
 For any $\varepsilon_0\in(0,1/4)$ there is a positive constant $\varepsilon_1=\varepsilon_1(\delta_1,r,\delta_0,K,\varepsilon_0)$ such that $\Xi_h$ gives $C^1$ diffeomorphism from $\Omega$ to $\Omega^h$ provided that
\[
    \|Lh\|_{L^\infty(\Omega)}\le\varepsilon_1, \quad\|\nabla Lh\|_{L^\infty(\Omega)}\le\varepsilon_0.
\]
Moreover, $\varepsilon_1$ can be taken so that
\begin{equation} \label{EEPhi}
    |\nabla_x\Phi_h|(x)
    \le \left|Lh(y)\right|\left|\nabla_{\hat{\Gamma}}\mathbf{n}\left(\Pi(y)\right)\right|
    +4\left(\frac{\|\chi'\|_\infty}{\delta_1}
    \left|Lh(y)\right|+\left|\nabla Lh(y)\right| \right),
\end{equation}
where $x=\Xi_h(y)$.
\end{thm}
\begin{proof}
It suffices to prove that $\|\nabla_x\Phi_h\|_\infty<1$.
 By Lemma~\ref{LInv} we see that
\begin{align}
    \nabla_x\Phi_h(x)
    &=\left(I-\frac{N_h+D_h}{K_h}\right) (I+R_h)^{-1}-I \notag \\
    &=\left(I-\frac{N_h+D_h}{K_h}\right) \left((I+R_h)^{-1}-I \right)
    +\left(I-\frac{N_h+D_h}{K_h} \right) - I \notag \\
    &=\left(I-\frac{N_h+D_h}{K_h}\right) \sum_{j=1}^\infty(-R_h)^j
    -\frac{N_h+D_h}{K_h} \notag \\
    &= \sum_{j=1}^\infty(-R_h)^j
    -\frac{N_h+D_h}{K_h} \sum_{j=0}^\infty(-R_h)^j \label{EInv}
\end{align}
provided that $\|R_h\|_\infty<1$ and $K_h>0$.
We take $\varepsilon_1$ small, i.e.,
\[
    0<\varepsilon_1\le\min\left(\frac{1}{4C_0}, \frac12\frac{1}{\|\nabla_{\hat{\Gamma}}\mathbf{n}\|_\infty}\right), \quad
    C_0=\frac{\|\chi'\|_\infty}{\delta_1}
\]
to conclude that
\begin{gather*}
    K_h\ge1-\frac14 -\varepsilon_0\ge\frac12, \\
    \|R_h\|_\infty\le\varepsilon_1\|\nabla_{\hat{\Gamma}}\mathbf{n}\|_\infty\le\frac12.
\end{gather*}
Note that $\left|(N_h+D_h)(y)\right|\le C_0\left|Lh(y)\right|+\left|\nabla Lh(y)\right|$.
 By \eqref{EInv}, we see that
\begin{align*}
    \left|\nabla_x\Phi_h(x)\right|
    &\le |Rh|+ \frac{1}{K_h} \left(C_0 \left|Lh(y)\right|+\left|\nabla Lh(y)\right|\right) \cdot2 \\
    &\le |R_h| + 2\cdot2 \left(C_0 \left|Lh(y)\right|+\left|\nabla Lh(y)\right|\right)
\end{align*}
since $\|R_h\|_\infty\le1/2$ implies
\[
    \sum_{j=1}^\infty \|R_h\|_\infty^j
    \le \sum_{j=1}^\infty (1/2)^j=1.
\]
This yields \eqref{EEPhi} since $|R_h|\le\left|Lh(y)\right|\left|\nabla_{\hat{\Gamma}}\mathbf{n}\left(\Pi(y)\right)\right|$.
 Thus
\[
    \left|\nabla_x\Phi_h(x)\right|
    \le\varepsilon_1\|\nabla_{\hat{\Gamma}}\mathbf{n}\|_\infty
    +4C_0\varepsilon_1+4\varepsilon_0.
\]
We take  $\varepsilon_1$ small so that
\[
    \varepsilon_1 \|\nabla_{\hat{\Gamma}}\mathbf{n}\|_\infty
    +4C_0\varepsilon_1
    <1-4\varepsilon_0
\]
to get
\[
    \sup_{x\in\Xi_h(\Omega)}\left|\nabla_x\Phi_h(x)\right| < 1.
\]
We now conclude that $\Xi_h$ is injective and its Jacobi matrix $\nabla\Xi_h$ is invertible so it is a $C^1$ diffeomorphism from $\Omega$ to its image.
 Since it is a diffeomorphism, the boundary $\hat{\Gamma}$ maps to $\Gamma_h=\left\{ y+h(y)\mathbf{n}(y)\bigm| y\in\hat{\Gamma}\right\}$, so $\Omega$ maps to $\Omega^h$.
\end{proof}
\begin{cor} \label{CInv}
Assume the same hypotheses of Theorem~\ref{TInv} concerning $\Omega$.
 Assume that $h_i$ ($i=1,2$) satisfies assumption of $h$ in Theorem~\ref{TInv} with constants $\varepsilon_0$, $\varepsilon_1$.
 Then
\begin{equation} \label{EDiff}
    \left| \nabla_x \Phi_{h_1}(x)-\nabla_x\Phi_{h_2}(x)\right|
    \le C_1 \left(\left| L(h_1-h_2)(y)\right|
    + \left|\nabla L(h_1-h_2)(y)\right|\right)
\end{equation}
with some constant $C_1$ depending only on $\varepsilon_0$, $\varepsilon_1$, $\delta_1$, $\|\chi'\|_\infty$ and $\Omega$ through $(r,\delta_0,K)$.
\end{cor}
\begin{proof}
By \eqref{EInv} we observe that
\begin{equation} \label{EInv2}
    \nabla_x \Phi_{h_i}(x)
    =-R_{h_i}I_{h_i}
    + \frac{N_{h_i}+D_{h_i}}{K_{h_i}}R_{h_i}I_{h_i}, \quad
    I_{h_i}=(I+R_h)^{-1}.
\end{equation}
Since $R_{h_1}R_{h_2}=R_{h_2}R_{h_1}$, matrices $I_{h_1}$, $I_{h_2}$, $R_{h_1}$, $R_{h_2}$ are commutative.
 We thus observe that
\[
    I_{h_1}-I_{h_2}
    =(R_{h_2}-R_{h_1})I_{h_1}I_{h_2}.
\]
By the choice of parameters as in Theorem~\ref{TInv}, $I_{h_i}$ is estimated as
\[
    \|I_{h_i}\|_\infty \le 1+1=2.
\]
We now obtain that
\begin{equation}
\begin{aligned} \label{EED1}
    |R_{h_1}I_{h_i}-R_{h_2}I_{h_2}|
    &\le |R_{h_1}-R_{h_2}||I_{h_1}|
    +|R_{h_2}||I_{h_1}-I_{h_2}| \\
    &\le 2|R_{h_1}-R_{h_2}|+4|R_{h_1}-R_{h_2}| \\
    &\le 6\left|\nabla\mathbf{n}\left(\Pi(y)\right)\right|
    \left|L(h_1-h_2)\right|(y),\quad
    x=\Xi_h(y).
\end{aligned}
\end{equation}
By definition, we see that
\begin{align}
    &|N_{h_1}-N_{h_2}|
    \le \frac{1}{\delta_1}\|\chi'\|_\infty L(h_1-h_2), \label{EED2} \\
    &|D_{h_1}-D_{h_2}|
    \le \left|\nabla L(h_1-h_2)\right|. \label{EED3}
\end{align}
Since
\[
    K_{h_1}^{-1}-K_{h_2}^{-1}
    =(K_{h_2}-K_{h_1})/(K_{h_1}K_{h_2})
\]
and
\[
    |K_{h_i}^{-1}|\le2,
\]
by the choice of the parameter, we observe that
\begin{equation}
\begin{aligned}
\label{EED4}
    |K_{h_1}^{-1}-K_{h_2}^{-1}|
    &=\left|K_{h_1}^{-1}K_{h_2}^{-1}(K_{h_2}-K_{h_1})\right|
    \le 4|K_{h_2}-K_{h_1}| \\
    &\le 4\left( \frac{1}{\delta_1}\|\chi'\|_\infty \left|L(h_1-h_2)\right|
    +\left|\nabla L(h_1-h_2)\right|\right).
\end{aligned}
\end{equation}
We collect these estimates \eqref{EED1}--\eqref{EED4} to conclude the desired estimates for the difference $|\nabla_x\Phi_{h_1}-\nabla_x\Phi_{h_2}|$.
\end{proof}

%%%%%%%%%%%%%%%%%%%%%%%%%%%%%%%%%%%%%%%%%%%%%%%%%
\subsection{Time derivative of $\Phi_h$} \label{SSTi} % Section 3.4

We are interested in time derivative of $\Phi_h$ when $h$ depends on time.
 In the case $h$ depends on time we consider $\Xi_{h(t)}(y)$.
 In this subsection, we also write it $\Xi_h(y,t)$.
 We use similar notation for $Z_{h(t)}$ and $\Phi_{h(t)}$.
\begin{lem} \label{LTDH}
Let $(p,q)$ satisfy $1\le p,q\le\infty$.
 Assume that
\[
    \partial_t(Lh)\in L^p\left(0,T;W^{2,q}(\Omega)\right).
\]
Then
 \[
    \partial_t\Phi_h(x,t) = \chi\left(\frac{d(y)}{S_1}\right)
    \partial_t(Lh)K_h^{-1}\mathbf{n}\left(\Pi(y)\right)
\]
so that
\[
    \partial_t\Phi_h \in L^p\left(0,T;W^{2,q}(\Omega)\right)
\]
provided that $K_h^{-1}\in L^p\left(0,T;W^{2,\infty}(\Omega)\right)$.
\end{lem}
\begin{proof}
Applying $\partial_t$ to the identity $Z_h\left(\Xi_h(y,t),t\right)=y$, we obtain
\[
    (\partial_t\Phi_h)(x,t)=(\partial_t Z_h)(x,t)
    =-(\nabla_x Z_h)(x,t) \partial_t\Xi_h.
\]
By Lemma~\ref{LInv},
\[
    \nabla_x Z_h= \left(I-\frac{N_h+D_h}{K_h}\right)(I+R_h)^{-1}.
\]
Since $\mathbf{n}\cdot\nabla\mathbf{n}=0$ so that $R_h\cdot\mathbf{n}\left(\Pi(y)\right)=0$, we see that $(I+R_h)^{-1}\mathbf{n}\left(\Pi(y)\right)=\mathbf{n}\left(\Pi(y)\right)$.
 Thus
\[
    \nabla_x Z_h(x,t)\cdot\mathbf{n}\left(\Pi(y)\right)
    = \left(I-\frac{N_h+D_h}{K_h}\right)\mathbf{n}.
\]
Since
\[
    (N_h+D_h)\mathbf{n}\left(\Pi(y)\right)
    = (K_h-1)\mathbf{n}\left(\Pi(y)\right),
\]
we now observe that
\[
    \nabla_x Z_h(x,t)\cdot\mathbf{n}\left(\Pi(y)\right)
    = K_h^{-1}\mathbf{n}\left(\Pi(y)\right).
\]
Since
\[
    \partial_t \Xi_h=\chi\left(\frac{d}{\delta_1}\right)
    \partial_t(Lh)(y,t)\mathbf{n}\left(\Pi(y)\right),
\]
we now conclude that
\[
    (\partial_t \Xi_h)(x,t)
    =-K_h^{-1}\mathbf{n}\left(\Pi(y)\right)
    \chi\left(\frac{d}{\delta_1}\right)\partial_t(Lh).
\]
The property $\partial_t\Phi_h\in L^p\left(0,T;W^{2,q}(\Omega)\right)$ follows from this formula and our assumptions.
\end{proof}

%%%%%%%%%%%%%%%%%%%%%%%%%%%%%%%%%%%%%%%%%%%%%%%%%
\subsection{Estimates of difference of $\Phi_h$} \label{SSDiff} % Subsection 3.5

We next apply Corollary~\ref{CInv} to get an estimate of the difference when $h$ is time-dependent.
 The function we shall consider for $h$ in the next section is
\[
    X_T:=L^p( 0,T;W^{3-1/q,q}(\hat{\Gamma}) )
    \cap W^{1,p}( 0,T;W^{2-1/q,q}(\hat{\Gamma}) ).
\]
By the property of the lifting operator $L$ stated in Lemma~\ref{LLi}, $h\in X_T$ implies that
\[
    Lh\in L^p\left(0,T;W^{3,q}(\Omega)\right)
    \cap W^{1,p}\left(0,T;W^{2,q}(\Omega)\right)
    =: H_T
\]
and $h\mapsto Lh$ is continuous.
 By definition of Besov space (or standard interpolation), if $h\in X_T$, then
\[
    h\in C( [0,T),B_{q,p}^{3-1/p-1/q}(\hat{\Gamma}) ).
\]
By the Sobolev embedding for the Besov space (see e.g.\ \cite{Saw}), we know that
\[
    \|h\|_{W^{2,\infty}}\le C\|h\|_{B_{q,p}^{3-1/p-1/q}(\hat{\Gamma})}
\]
provided that $2/p+d/q<1$.
 Similarly, if $Lh\in H_T$, then
\[
    Lh\in C( [0,T),B_{q,p}^{3-2/p}(\Omega) )
\]
so that
\[
    \|Lh\|_{W^{2,\infty}}
    \le C\|Lh\|_{B_{q,p}^{3-2/p}(\Omega)}
\]
for $2/p+d/q<1$.
 Thus if $\|h\|_{X_T}\le\varepsilon$, then
\[
    \|Lh\|_{W^{2,\infty}}
    \le c\varepsilon
\]
with some $c$ depending only on $C^3$-regularity of $\hat\Gamma$.
 We have proved the following statement.
\begin{prop} \label{PSL}
Let $\Omega$ be a $C^3$ splash domain with compact boundary of type $(r,\delta_0,K)$.
 If $\|h\|_{X_T}\le\varepsilon$ for sufficiently small $\varepsilon>0$, i.e., $\varepsilon<\varepsilon_*=\varepsilon_*(r,\delta_0,K)$, then \eqref{EEPhi} holds provided that $2/p+d/q<1$.
\end{prop}

We next consider the difference estimate when $h\in X_T$.
\begin{lem} \label{LDif}
Under the same hypotheses of Proposition~\ref{PSL} concerning $\Omega$, $p$, $q$ and $\varepsilon_*$.
 Assume that $\|h_i\|_{X_T}<\varepsilon_*$ for $i=1,2$ with $h_1(0)=h_2(0)$.
 Then
\[
    \|\nabla_x\Phi_{h_1(t)}-\nabla_x\Phi_{h_2(t)}\|_{L^\infty\left(0,T;L^\infty(\Omega)\right)}
    \le CT^{1/p'}C\left\|\partial_t(h_1-h_2)\right\|_{L^p( 0,T;W^{2-1/q,q}(\hat{\Gamma}) )}
\]
where $1/p'+1/p=1$ and $p\in(2,\infty]$.
\end{lem}
\begin{proof}
This follows from Corollary~\ref{CInv} by estimating
\begin{align*}
    \|L\bar{h}\|_{W^{1,\infty}(\Omega)}(t)
    &\le \int_0^T \|\partial_s L\bar{h}\|_{W^{1,\infty}}\,ds \\
    &\le C\int_0^T \|\partial_s L\bar{h}\|_{W^{2,q}}\,ds \quad \text{(since}\ q>d\text{)} \\
    &\le CT^{1/p'} \int_0^T \|\partial_s L\bar{h}\|_{W^{2,q}}^p\,ds
\end{align*}
for $\bar{h}=h_1-h_2$ with $C$ depending only on $\Omega$.
\end{proof}

%%%%%%%%%%%%%%%%%%%%%%%%%%%%%%%%%%%%%%%%%%%%%%%%%
\section{Local well-posedness in a domain with $\delta$-wing and the existence of collapse} \label{SL} % Section 4

In this section, we shall state a locally-in-time unique existence theorem for the free boundary problem to \eqref{ENS}, \eqref{EIn} in a domain with $\delta$-wing.
 We then apply this result to show the existence of collapse stated in Theorem~\ref{TMain}.

%%%%%%%%%%%%%%%%%%%%%%%%%%%%%%%%%%%%%%%%%%%%%%%%%
\subsection{Local well-posedness} \label{SSL} % Section 4.1

Let $\Omega$ be a $C^k$ ($k\ge2$) splash domain in $\mathbb{R}^d$ with compact boundary.
 Let $\Omega_\delta$ be a domain with $\delta$-wing constructed from $\Omega$.
 Let $\hat{\Gamma}$ be the boundary of $\Omega$ in $\Omega_\delta$.
 We use a normal coordinate in a $\delta_1$-neighborhood of $\hat{\Gamma}$ of the form \eqref{ECor}.
 As noted in a previous section, this is $C^{k-1}$ coordinate change.
 For a given function $h\in C^1(\hat{\Gamma})$ with $\|h\|_\infty\le\delta_1$ we define a domain
\[
    \Omega^h=\left\{ x=y+s\mathbf{n}(y)\in\Omega_\delta \bigm|
    y\in\hat{\Gamma},\ 
    -\delta_1\le s<h(y)\right\}
    \cup \left\{ x\in\Omega \bigm|
    \operatorname{dist}(x,\partial\Omega)<
    \delta_1\right\}
\]
in $\Omega_{\delta_1}$.
 As we observed in the previous section, the Hanzawa transform $\Xi_h$ is a $C^{k-1}$ diffeomorphism from $\Omega$ to $\Omega^h$ (up to the boundary) provided that $\|\nabla h\|_\infty$ is small, say $\|\nabla h\|_\infty\le\varepsilon_0$.
 We notice that constants $\delta_1$, $\varepsilon_0$ depend on $\Omega$ only through its $C^2$-regularity of $\hat{\Gamma}$ i.e., the type $(r,\delta_0,K)$ of the boundary for $k=2$.
 For a function $f$ defined on $\Omega^h$ let $f_h^\#$ denote its pull-back by $\Xi_h$, i.e.,
\[
    f_h^\#(y)=f\left(\Xi_h(y)\right), \quad
    y\in\Omega.
\]
Of course,
\[
    f(x)=f_h^\#\left(Z_h(x)\right), \quad
    x\in\Omega,
\]
where $Z_h=\Xi_h^{-1}$.
 If $h$ is time dependent and $f$ is time dependent, we still write
\[
    f_h^\#(y,t)=f\left(\Xi_{h(t)}(y),t\right), \quad
    y\in\Omega^h, \quad
    t\in(0,T).
\]
For a given function $h_0$ on $\hat{\Gamma}$ with $\|h_0\|_\infty\le\delta_1$, let $\partial_\delta\Omega^{h_0}$ denote the boundary of $\Omega^{h_0}$ in $\Omega_\delta$ and $\mathbf{n}_{h_0}$ denotes its exterior unit normal vector field in $\Omega_\delta$.

We prepare function spaces.
 For $p,q\in(1,\infty)$ we set
\begin{align*}
    &X_T:=L^p( 0,T;W^{3-1/q,q}(\hat{\Gamma}) )
    \cap W^{1,p}( 0,T;W^{2-1/q,q}(\hat{\Gamma}) ) \\
    &Y_T:=L^p\left(0,T;W^{2,q}(\Omega)\right)
    \cap W^{1,p}\left(0,T;L^q(\Omega)\right).
\end{align*}
The trace space $W^{m-1/q,q}(\hat{\Gamma})$ can be interpreted as a Besov space $B_{q,q}^{m-1/q}$.
 Let
\[
    \hat{W}_0^{1,q}(\Omega)
    = \left\{ \varphi\in L_{loc}^q(\Omega) \bigm|
    \nabla\varphi\in L^q(\Omega),\ \varphi|_{\hat{\Gamma}}=0 \right\}.
\]
We set
\[
    Z_T:=L^p( 0,T;W^{1,q}(\Omega)
    + \hat{W}_0^{1,q}(\Omega) ).
\]
\begin{thm} \label{TLW}
Let $\Omega$ be a $C^3$ splash domain in $\mathbb{R}^d$ with compact boundary (of type $(r,\delta_0,K)$).
 Assume that $2/p+d/q<1$, $2<p<\infty$, $d<q<\infty$.
 (Let $\varepsilon_0$ and $\delta_1$ be given as above.)
 Then for $B>0$, there exists (small) $T=T(B,p,q,d,r,\delta_0,K)>0$ and (small) $\varepsilon=\varepsilon(p,q,d,r,\delta_0,K)>0$ such that if
\[
    \|h_0\|_{B_{q,p}^{3-1/p-1/q}(\hat{\Gamma})}\le\varepsilon, \quad
    \|u_0\|_{B_{q,p}^{2(1-1/p)}}\le B
\]
and $u_0$ satisfies the compatibility condition
\begin{equation} \label{ECP}
    \operatorname{div}u_0=0
    \quad \text{in}\quad \Omega^{h_0}
    \quad\text{and}\quad
    (I-\mathbf{n}_{h_0}\otimes \mathbf{n}_{h_0})\mathbb{D}(u_0)\mathbf{n}_{h_0}=0
    \quad\text{on}\quad \partial_\delta \Omega^h
\end{equation}
then there exists a unique solution $(u,\varpi,\Omega^{h(t)})$ to \eqref{ENS}, \eqref{EIn} in $(0,T)$ with $u|_{t=0}=u_0$ and $\Omega|_{t=0}=\Omega^{h_0}$ which satisfies
\[
    h\in X_T, \quad
    u_h^\#\in Y_T, \quad
    \varpi_h^\#\in Z_T
    \quad\text{and}\quad
    \|\nabla h\|_\infty\le\varepsilon_0, \quad
    \|h\|_\infty\le\delta_1
\]
(so that $\sup_{0\le t<T}\|\nabla\Phi_{h(t)}\|_\infty(t)<1$). 
 Moreover,
\[
    \|h\|_{X_T}
    + \|u_h^\#\|_{Y_T}
    + \|\varpi_h^\#\|_{Z_T} \le CB
\]
with $C=C(p,q,d,r,\delta_0,K)$.
\end{thm}

The compatibility condition $\operatorname{div}u_0=0$ in \eqref{ECP} is rather formal notation.
 As we discuss later, it means that $u_0$ belongs to a solenoidal space $J_q (\Omega^{h_0})$ which is defined in Subsection~\ref{SSMax}.

When the boundary $\Omega$ is embedded, this is contained in a general result \cite[Theorem~6.1]{Sh} as a special case.
 He considered a general uniformly $C^3$ domain with embedded boundary not necessarily compact.
 As we later sketch his proof, his result extends to a splash domain.
 Here we state the case that $\hat{\Gamma}$ is compact to simplify the statement.\\

\vspace{0.7em}
\noindent{\bf Comparison of assumptions.}
 In \cite[Theorem~6.1]{Sh}, it is assumed that inside of $\Omega$ has a finite covering \cite[Definition~2.3]{Sh}.
 Of course, this assumption is trivially fulfilled if $\Omega$ is an exterior or a bounded domain.
 A key assumption is unique solvability of weak Dirichlet problem for index $q$ and $q'=q/(q-1)$, which is the conjugate exponent of $q$.
 This assumption is fulfilled for a $C^1$ splash domain with compact boundary.
\begin{lem} \label{LWD}
Let $\Omega$ be a $C^1$ splash domain with compact boundary $\hat{\Gamma}$ in $\Omega_\delta$.
 Assume that $1<q<\infty$.
 Then for any $f\in\left(L^q(\Omega)\right)^d$, there exists a unique solution $u\in\hat{W}_{q,0}^1(\Omega)$ satisfying
\[
    \int_\Omega \nabla u\cdot\nabla\varphi\,dx
    = \int_\Omega f\cdot\nabla\varphi\,dx
    \quad\text{for all}\quad
    \varphi\in\hat{W}_{q',0}^1(\Omega).
\]
Moreover the operator $f\mapsto\nabla u$ is bounded in $\left(L^q(\Omega)\right)^d$, i.e.,
\[
    \|\nabla u\|_{\left(L^q(\Omega)\right)^d}
    \le C\|f\|_{\left(L^q(\Omega)\right)^d}
\]
with some $C=C(\Gamma)$.
\end{lem}
If $\Omega$ is a $C^1$ domain with embedded compact boundary, this is proved for example in \cite[Theorem~7.4.3]{PS} by localization.
 The embeddedness of the boundary is not invoked.
 In \cite[Theorem~7.4.3]{PS}, non-homogeneous data is also considered and Dirichlet condition is given on some part of the boundary which may not be on the whole boundary.
 There is another proof of Lemma~\ref{LWD} with more regularity of the boundary based on the Stokes equations which is given by \cite[Theorem~3.2]{ShCPAA}.
 This type of problem is related to the classical Helmholtz decomposition if there is no Dirichlet condition which goes back to \cite[P.~45, Theorem~1.2]{SS}.

In \cite[Theorem~6.1]{Sh}, it is assumed that
\[
    \kappa\in W^{1-1/q,q}(\hat{\Gamma})
    \quad\text{or}\quad
    \nabla(L\kappa)\in L^q(\Omega),
\]
where $\kappa$ is the mean curvature of $\hat{\Gamma}$.
 Since $\Gamma$ is compact and $C^3$, $\kappa\in C^1(\hat{\Gamma})$ so $\kappa\in W^{1-1/q,q}(\hat{\Gamma})$ is fulfilled.

In \cite[Theorem~6.1]{Sh} the compatibility condition is written for a function $u_{0_{h_0}}^\#$ but as we expected it is the same.
 Thus all assumptions in \cite[Theorem~6.1]{Sh} are fulfilled (except embeddedness of $\Gamma$).

\vspace{0.7em}
\noindent{\bf Estimates in original coordinates.}
 We notice that estimates for $u_h^\#$, $\varpi_h^\#$ yield this estimate for original variables, for example,
\begin{equation} \label{EOri}
    \left(\int_0^T \|u\|_{W^{2,q}(\Omega^{h(t)})}^p\, dt
    + \int_0^T \|\partial_t u\|_{L^q(\Omega^{h(t)})}^p\, dt \right)^{1/p}
    \le C'B
\end{equation}
with $C'=C'(p,q,d,r,\delta_0,K)$.
 We notice that
\[
    u(x,t)=u_h^\# \left(Z_{h(t)}(x),t \right).
\]
In general, for a function $f^\#$ defined on $\Omega\times(0,T)$, we set
\[
    f(x,t)=f^\# \left(Z_{h(t)}(x),t \right).
\]
Then
\begin{align*}
    \partial_{x_i}f(x,t) &= \sum_{\ell=1}^d \partial_{y_\ell}f^\#\left(Z_h(x),t\right) \partial_{x_i}Z_h^\ell(x), \\
    \partial_{x_i}\partial_{x_j}f(x,t) &= \sum_{1\le\ell,k\le d} (\partial_{y_\ell}\partial_{y_k}f^\#)\left(Z_h(x),t\right)
    \partial_{x_i}Z_h^\ell \partial_{x_j}Z_h(x)
    + \sum_{\ell=1}^d \partial_{y_\ell}f\left(Z_h(x),t\right)
    \partial_{x_i} \partial_{x_j} Z_h^\ell(x), \\
    \partial_t f(x,t) &= \sum_{\ell=1}^d \partial_{y_\ell}f^\#\left(Z_h(x),t\right) \partial_t Z_h^\ell
    +(\partial_t f^\#) \left(Z_h(x),t\right).
\end{align*}
If $h\in X_T$, then
\[
    Lh\in L^p\left(0,T;W^{3,q}(\Omega)\right)
    \cap W^{1,p}\left(0,T;W^{2,q}(\Omega)\right).
\]
Since $p>2$, by the Sobolev embedding, $Lh\in C\left([0,T],W^{2,q}(\Omega)\right)$.
 If $q>d$, by the Sobolev embedding, $Lh\in L^\infty\left(0,T;W^{1,\infty}(\Omega)\right)$.
 If $\|\nabla h\|_\infty$ is small, say $\|\nabla h\|_\infty\le\varepsilon_0$, then we are able to apply Theorem~\ref{TInv} to conclude
\begin{equation} \label{EFi}
    \sup_{0<t<T} \|\nabla_x\Phi_{h(t)}\|_{L^\infty} < 1
\end{equation}
since $\mathbf{n}$ is $C^1$ on $\hat{\Gamma}$ and $\hat{\Gamma}$ is compact.
 Since $\Omega$ is $C^3$ so that $\mathbf{n}$ is $C^2$, regularity of $\nabla^2 Z_h$ is the same as $\nabla^2 Lh$ by Lemma~\ref{LInv}.
 In particular,
\[
    \|\nabla^2 Z_h\|_{L^q},\ 
    \left\|\partial_t(\nabla^2 Z_h)\right\|_{L^q(\Omega^{h(t)})}
    \in L^p(0,T)
\]
which implies
\begin{equation} \label{ESec}
    \|\nabla^2 Z_h\|_{L^q(\Omega^{h(t)})}
    \in L^\infty(0,T).   
\end{equation}
In Theorem~\ref{TLW}, $u^\#=u_h^\#$ satisfies
\[
    u^\#\in L^p\left(0,T;W^{2,q}(\Omega)\right)
    \cap W^{1,p}\left(0,T;L^q(\Omega)\right)
\]
with $2<p<\infty$, $d<q<\infty$ satisfying $2/p+d/q<1$.
 By a standard interpolation
\[
    u^\#\in L^\infty( 0,T;B_{q,p}^{2(1-1/p)}(\Omega) ),
\]
which implies
\begin{equation}
\begin{aligned}\label{ELQ}
    \nabla u^\# \in L^\infty( 0,T;B_{q,p}^{1-2/p}(\Omega) )
    &\subset L^\infty( 0,T;B_{q,p}^{d/q}(\Omega) ) \\
    &\subset L^\infty( 0,T;L^q(\Omega) )
\end{aligned}
\end{equation}
since $2/p+d/q<1$.
 Thus
\begin{align*}
    \|\nabla_y u\|_{L^q(\Omega^{h(t)})}
    &\le C\|\nabla u^\#(t)\|_{L^q(\Omega)}
    \|\nabla Z_{h(t)}\|_{L^\infty(\Omega^{h(t)})} \\
    \|\nabla_y^2 u\|_{L^q(\Omega^{h(t)})}
    &\le C\|\nabla^2 u^\#(t)\|_{L^q(\Omega)}
    \|\nabla Z_{h(t)}\|_{L^\infty(\Omega^{h(t)})} \\
    &+ C\|\nabla u^\#(t)\|_{L^q(\Omega)}
    \|\nabla^2 Z_{h(t)}\|_{L^\infty(\Omega^{h(t)})}.
\end{align*}
The constant $C$ appears from a bound for $\operatorname{det}(\nabla\Xi_{h(t)})$;
 this bound is independent of $t$ and $x$ since $\left\|\nabla h(t)\right\|_\infty$ is bounded.
 We estimate
\[
    \|\nabla_y^2 u\|_{L^q(\Omega^{h(t)})} 
    \le C\|\nabla^2 u^\#(t)\|_{L^q(\Omega)}M
    +C\sup_{0<\tau<T} \|\nabla u^\#(\tau)\|_{L^q(\Omega)}
    \cdot \|\nabla^2 Z_h\|_{L^q(\Omega^{h(t)})}
\]
so that
\begin{align*}
    \left( \int_0^T \|\nabla_y^2 u\|_{L^q(\Omega^{h(t)})}^p\,dt \right)^{1/p}
    &\le CM\left( \int_0^T \|\nabla^2 u^\#(t)\|_{L^q(\Omega)}^p\,dt \right)^{1/p} \\
    &+CN\left( \int_0^T \|\nabla^2 Z_h\|_{L^q(\Omega^{h(t)})}^p\,dt \right)^{1/p}
\end{align*}
where $M=\sup_{0<t<T} \|\nabla Z_{h(t)}\|_{L^\infty(\Omega^{h(t)})}$, $N=\sup_{0<\tau<T} \|\nabla u^\#(t)\|_{L^\infty(\Omega)}$.
 By \eqref{EFi}, \eqref{ESec} and \eqref{ELQ}.
 This yields that $\|\nabla^2 u\|_{L^q(\Omega^{h(t)})}\in L^p(0,T)$ with necessary estimates.
 Similar observation implies that $\|\partial_t u\|_{L^q(\Omega^{h(t)})}\in L^p(0,T)$ with necessary estimates; see Lemma~\ref{LTDH}. 

We now obtain \eqref{EOri}.

%%%%%%%%%%%%%%%%%%%%%%%%%%%%%%%%%%%%%%%%%%%%%%%%%
\subsection{Construction of a family of domains} \label{SSC} % Section 4.2

For a splash domain $\Omega$ in $\mathbb{R}^d$, we construct a family of domain $\Omega_{0,\eta}$ with $\bar{\Omega}_{0,\eta}\subset \Omega$ such that at some self-intersection point of $\hat{\Gamma}$, $\partial\Omega_{0,\eta}$ is close to this point in both directions keeping necessary bound.
 For brevity, we set $B_r=B_r^{d-1}(0)$ in this subsection.
\begin{prop} \label{PCh}
    Let $f\in C^k(\overline{B_r})$ satisfying \eqref{EAF1} and \eqref{EAF2}, with $k\ge1$.
    Let $g\in C^k(\overline{B_r})$ satisfy $f+\eta_0<g<r$ on $\overline{B_r}$ for some $\eta_0>0$.
    Then for any $\eta\in(0,\eta_0)$ there exists a function $f_\eta\in C^k(\overline{B_r})$ such that 
\[
	f_\eta(x) = \left\{
\begin{array}{ll}
     f(x)+\eta &\text{in}\quad \overline{B_{r/2}} \\
     g(x) &\text{in}\quad \overline{B_r}\setminus B_{3r/4}
\end{array}
\right.,\quad f_\eta<g \ \text{on}\ \overline{B_r}
\]
and
\[
    \sup_{\eta\in(0,\eta_0)}K_{f_\eta} \le c_r(K_f+K_g)
\]
with $c_r$ depending only on $r$.
\end{prop}
\begin{proof}
    Let $\varphi\in C_c^\infty(B_r)$ be a cut-off function of $B_{r/2}$ such that $\varphi=1$ on $B_{r/2}$ and $\varphi=0$ on $B_r\setminus B_{3r/4}$ with $0\leq\varphi\le1$ on $B_r$.
    We set
\[
    f_\eta=\varphi\left\{(f+\eta)-g\right\}+g
\]
and observe that $f_\eta$ satisfies all desired properties.
\end{proof}

We shall use notation $C_{r,f}$ in Section~\ref{SD}.
 Let $\Omega$ be a splash domain in $\mathbb{R}^d$ with $C^k$ ($k\ge2$) boundary $\hat{\Gamma}$ of type $(r,\delta_0,K)$.
 Let $P\in\Gamma=\partial\Omega$ be a point of self-intersection, i.e., $\pi^{-1}(P)$ is not a singleton so that $\iota_{1,P}$ and $\iota_{2,P}$ are different.
 Let $\delta_1>0$ be a number such that normal coordinate is available in $\delta_1$-tubular neighborhood of $\hat{\Gamma}$ in $\Omega_\delta$.
 (The existence of $\delta_1>0$ is guaranteed since $k\ge2$.)
 We consider a domain
\[
    \tilde{D}_{\delta_2} = \left\{ x\in\Omega \bigm|
    \operatorname{dist}(x,\Gamma)>\delta_2 \right\}
\]
for $\delta_2\in(0,\delta_1)$.
 If $\delta_2<\delta_1/2$ is sufficiently small,
\[
    \tilde{D}_{\delta_2} \cap \iota_{j,P}(C_{r,f_j})
    = \iota_{j,P}(C_{r,g_j}) \quad\text{(}j=1,2\text{)}
\]
for some function $g_j>f_j$ in $\overline{B_r}$, where
\[
    \Omega\cap i_{1,P}(C_r)
    =\iota_{1,P}(C_{r,f_1})
    \cup\iota_{2,P}(C_{r,f_2})
    \cup\bigcup_{j\equiv3}^m D_j,
\]
where $D_j$ is an open set such that $\partial D_j\not\ni P$.
 We modify $g_j$ by Proposition~\ref{PCh} so that we define $f_\eta$.
 For sufficiently small $\eta>0$
\[
    \Omega_{(\eta)}
    =\left(\tilde{D}_{\delta_2}\setminus\iota_{1,P}(C_r)\right) 
    \cup\iota_{1,P}(C_{r,f_{1,\eta}})
    \cup\iota_{2,P}(C_{r,f_{2,\eta}})
\]
for $\eta$ sufficiently small.
 Note that $\iota_{1,P}(C_r)=\iota_{2,P}(C_r)$ since $Q_{j,P}(C_r)$ for $j=1,2$ is the same.
 We set $\eta_0>0$ so that $f_{j,\eta}<g_j$ for $\eta<\eta_0,$ $j=1,2$.
 By using normal coordinate $\Omega_{(\eta)}$ can be written as
\[
    \Omega_{(\eta)}
    =\left\{ x=y+s\mathbf{n}(y)\in\Omega\setminus \tilde{D}_{\delta_2} \bigm|
    -\delta_2\le s\le h_\eta\ (<0),\ 
    y\in\hat{\Gamma} \right\}
    \cup \tilde{D}_{\delta_2}
\]
some function $h_\eta:\hat{\Gamma}\to(-\delta_2,0)$.
\begin{prop} \label{PAD}
If the immersed boundary $\hat{\Gamma}$ is uniformly $C^k$ ($k\ge2$) of type $(r,\delta_0,K)$.
Then for $\eta\in (0,\eta_0)$ the height function $h_\eta: \hat{\Gamma} \to (-\delta_2,0)$ that parameterizes $\partial \Omega_{(\eta)}$ over $\hat{\Gamma}$ has a bound with its $k$th derivatives on $\hat{\Gamma}$.
Moreover, this bound depends only on $(r,\delta_0,K)$.
\end{prop}
\begin{proof}
Since $h_\eta=-\delta_2$ outside $\iota_{1,P}(C_r)$, the Hausdorff distance $d_H(\mathcal{N} \hat{\Gamma}, \mathcal{N} \partial \Omega_{(\eta)})$ between normal bundles $\mathcal{N} \hat{\Gamma} := \{ (P, \mathbf{n}(P)) \bigm| P \in \hat{\Gamma} \}$ and $\mathcal{N} \partial \Omega_{(\eta)} := \{ ( y, \mathbf{n}_{\partial \Omega_{(\eta)}}(y) ) \bigm| y \in \partial \Omega_{(\eta)} \}$ can be controlled by a constant multiple of $r$ provided $\delta_1<r^2$.
Hence, when $r$ is sufficiently small, the embedded boundary $\partial \Omega_{(\eta)}$ can be parameterized over $\hat{\Gamma}$ by a $C^k$ height function, see e.g.\ \cite[Section~2.3.3]{PS}.
\end{proof}
\begin{cor} \label{PBC}
If $k\ge2$, $k-2$th derivative of the mean curvature $\kappa$ of $\Omega_{(\eta)}$ has a bound depending only on $(r,\delta_0,K)$.
\end{cor}
\begin{proof}
This is direct consequence of Proposition~\ref{PCh} and the construction that $h_\eta=-\delta_2$ outside $\iota_{1,P}(C_r)$.
\end{proof}

\begin{cor} \label{PRC1}
Assume that $\hat{\Gamma}$ is compact.
 In other words, $\Omega$ is either an exterior or a bounded domain.
 If $k=3$, then $h_\eta\in B_{q,p}^{3-1/p-1/q}(\hat{\Gamma})$ and
\[
    \|h_\eta\|_{B_{q,p}^{3-1/p-1/q}} \le C_0
\]
with $C_0$ independent of $\eta\in(0,\eta_0)$.
 Here $p,q\in(1,\infty)$.
\end{cor}
\begin{proof}
If $\hat{\Gamma}$ is compact, $C^3(\hat{\Gamma})\subset B_{p,q}^{3-\alpha}$ for any $\alpha\in(0,3)$ and $p,q\in(1,\infty)$.
 Then Proposition~\ref{PAD} completes the proof.
\end{proof}

%%%%%%%%%%%%%%%%%%%%%%%%%%%%%%%%%%%%%%%%%%%%%%%%%
\subsection{Construction of initial velocity} \label{SSVe} % Section 4.3

Let $\Omega$ be a $C^k$ ($k\ge2$) splash domain with compact boundary.
 We would like to construct an initial velocity on $\Omega_{(\eta)}$ having necessary uniform bound.
\begin{prop} \label{PIV}
For a given $b\in W^{2-1/q,q}(\partial\Omega_{(\eta)})$ satisfying $\int_{\partial\Omega_{(\eta)}}b\,d\mathcal{H}^{n-1}=0$, there exists a family of solenoidal velocity field $u_{0\eta}\in W^{2,q}(\Omega_{(\eta)})$ satisfying the compatibility condition on $\partial\Omega_{(\eta)}$ such that
\begin{gather}
u_{0\eta}\cdot\mathbf{n}_{\partial\Omega_{(\eta)}}=b, \notag \\
\|u_{0\eta}\|_{W^{2,q}}\le C_1\|b\|_{W^{2-1/q,q}(\partial\Omega_{(\eta)})} \label{EEIV}
\end{gather}
with $C_1$ depending only on $q\in(1,\infty)$ and $C^2$-regularity of $\partial\Omega$, i.e., $(r,\delta_0,K)$ if $\partial\Omega$ is considered as type $(r,\delta_0,K)$.
\end{prop}
\begin{proof}
Let us first consider the case that $\Omega$ is bounded so that $\Omega_{(\eta)}$ is also bounded.
 We consider the Stokes problem
\begin{equation}\label{ESSL}
	\left\{
\begin{array}{rl}
    -\Delta u+\nabla\varpi=0,
    &\operatorname{div}u=0\ \text{in}\ \Omega_{(\eta)}, \\
    \left(\mathbb{D}(u)\mathbf{n}\right)_{\tan}=0,
    & u\cdot\mathbf{n}=b\ \text{on}\ \partial\Omega_{(\eta)}
\end{array}
\right.
\end{equation}
with $b$ satisfying $\int_{\partial\Omega_{(\eta)}}b\,d\mathcal{H}^{d-1}=0$.
 By \cite{SS71}, there is a unique solution $(u,\nabla\varpi)$ such that for $q\in(1,\infty)$
\[
    \|u\|_{W^{2,q}(\Omega_{(\eta)})}
    \le C_1\|b\|_{W^{2-1/q,q}(\partial\Omega_{(\eta)})}
\]
and $C_1$ depends only on $q$ and $C^2$-regularity of $\Omega_{(\eta)}$.
 Since $\Omega$ is at least $C^2$ domain so that $\Omega_{(\eta)}$ is uniformly $C^2$, dependence of $C_1$ is what we required.

In the case of an unbounded domain, we consider $\Omega_{(\eta)R}=\Omega_{(\eta)}\cap B_{R+5}(0)$ such that $\Omega_{(\eta)}^c\subset B_R(0)$ for $\eta\in(0,\eta_0)$.
 Let $u_R$ be the solution of \eqref{ESSL} with $\Omega_{(\eta)}$ replaced by $\Omega_{(\eta)R}$.
 We multiply cut-off function $\varphi\in C_c^\infty(0,R+5)$ such that $\varphi=1$ on $[0,R+2]$ and $\varphi=0$ on $[R+3,R+5)$ with $0\le\varphi\le1$.
 We consider
\[
    u_{R\varphi}(x) = \varphi\left(|x|\right) u_R(x).
\]
This function satisfies
\[
    \left( \mathbb{D}(u)\mathbf{n}\right)_{\tan}=0
    \quad\text{and}\quad
    u\cdot\mathbf{n}=b
    \quad\text{on}\quad
    \partial\Omega_{(\eta)}    
\]
with \eqref{EEIV}.
 Since
\[
    \operatorname{div}u_{R\varphi}(x)
    =\nabla\left(\varphi\left(|x|\right)\right)\cdot u_R(x)=:\psi
    \quad\text{in}\quad
    \Omega_{0,\eta}
\]
and $\nabla\varphi\left(|x|\right)$ is supported in $A=B_{R+4}(0)\setminus\overline{B_{R+1}(0)}$($=\Omega_{(\eta)}\cap A$) by $\Omega_{(\eta)}^c\subset B_R(0)$, we apply the stantard Bogovski's lemma to conclude that there exists a $v\in W^{2,q}(A)$ with $\operatorname{spt}v\subset A$ satisfying
\[
    \operatorname{div}v=\psi
    \quad\text{in}\quad A.
\]
Moreover, $\psi\mapsto v$ is a continuous linear operator
\[
    \|v\|_{W^{2,q}(A)}
    \le C\|\psi\|_{W^{1,q}(A)}
\]
with $C$ depending only on $q$ and $R$; see P.~Galdi's textbook \cite{Gal}.
 (We need compatibility condition $\int_A\psi\,dx=0$ which is automatically fulfilled since $\psi=\operatorname{div}u_{R\varphi}$.)
 We set
\[
    u_0:=u_{R\varphi}-v
\]
to get a desired vector field.
\end{proof}

%%%%%%%%%%%%%%%%%%%%%%%%%%%%%%%%%%%%%%%%%%%%%%%%%
\subsection{Collapsing bubble} \label{SSCB} % Section 4.4

Let $\Omega$ be a $C^k$ ($k\ge2$) splash domain in $\mathbb{R}^d$ with compact boundary $\hat{\Gamma}$ in $\Omega_\delta$.
 For a given $\eta\in[0,\eta_0)$ we construct $\Omega_{(\eta)}$ as in Section~\ref{SSC} and write
\[
    \Omega_{(\eta)}=\Omega^{h_{0\eta}}
\]
for some function $h_{0\eta}:\hat{\Gamma}\to(-\delta_2,0)$.
 For a function $h_\eta:\hat{\Gamma}\times[0,T)\to(-\delta_2,\delta_2)$ with $h_\eta|_{t=0}=h_{0\eta}$, we consider an evolving family of domain
\[
    \Omega_\eta(t) := \Omega^{h_\eta(t)}.
\]
By definition, its initial configuration $\Omega_\eta(0)=\Omega_0^{h_{0\eta}}$, which we also write $\Omega_{0,\eta}$.
 Let $\|\cdot\|_{C^\mu}$ denote the H\"older norm on $\hat{\Gamma}\times[0,T)$ with exponent $\mu\in(0,1)$.
\begin{prop} \label{PCol}
Let $P\in\Gamma$ be a point of self-intersection.
 Assume that $h_{0\eta}$ is taken as in Section~\ref{SSC}.
 Assume that $h_\eta,\nabla h_\eta\in C\left(\hat{\Gamma}\times[0,T)\right)$ with
\[
    \|h_\eta\|_\infty\le\delta_2.
\]
Assume that the normal velocity $V$ of $\Omega_\eta(t)$ satisfies
\[
    \|V\|_{C^\mu}\le A_1
\]
with some constant $A_1$ as a function on $\hat{\Gamma}\times[0,T)$.
 Then, there exists $\eta_1\in(0,\eta_0)$ and $\delta_3>0$ such that for $\eta\in(0,\eta_1)$, $\partial\Omega_\eta(t)$ has a self-intersection point (near $P$) for some $t_0\in(0,T)$ provided that at $t=0$,
\begin{gather*}
    V=1 \quad\text{on}\quad \iota_{1,P}(\partial C_{r,f_{1,\eta}}\cap C_{r/2}), \\
    V=0 \quad\text{on}\quad \iota_{2,P}(\partial C_{r,f_{2,\eta}}\cap C_{r/2})
\end{gather*}
and
\[
        \|\nabla_{\hat{\Gamma}}h_\eta\|_\infty\le\delta_3.
\]
The number $\eta_1$ can be taken so that it only depends on $\delta_1$, $\delta_2$, $A_1$, $\mu$ and $C^2$ regularity of $\hat{\Gamma}$ (i.e., $(r,\delta_0,K)$).
\end{prop}
\begin{proof}
In a neighborhood $\iota_{1,P}(C_r)$ of $P$, $\Omega_{0,\eta}$ can be written as
\[
    C_{r/2,f_1+\eta}\cup \left(C_{r/2}\setminus\bar{C}_{r/2,-f_2-\eta}\right)
\]
by translation and notation.
 We may assume that $\Omega_\eta(t)$ can be written as
\begin{align*}
    &\Omega_\eta(t)\cap\iota_{1,P}(C_r)=\iota_{1,P}(C_{r,f_{1,\eta}(t)}) \\
    &\Omega_\eta(t)\cap\iota_{2,P}(C_r)=\iota_{2,P}(C_{r,f_{2,\eta}(t)})
    = \iota_{1,P}(C_r\setminus \bar{C}_{r,-f_{2,\eta}(t)})
\end{align*}
with
\[
    f_{1,\eta}(x',0)=f_1(x)+\eta, \quad
    f_{2,\eta}(x',0)=f_2(x')+\eta,
    \quad\text{for}\quad |x'|<r/2
\]
if $\delta_3$ is chosen small; see Figure~\ref{FCo}.
\begin{figure}[htb]
\centering
    \includegraphics[width=0.4\linewidth]{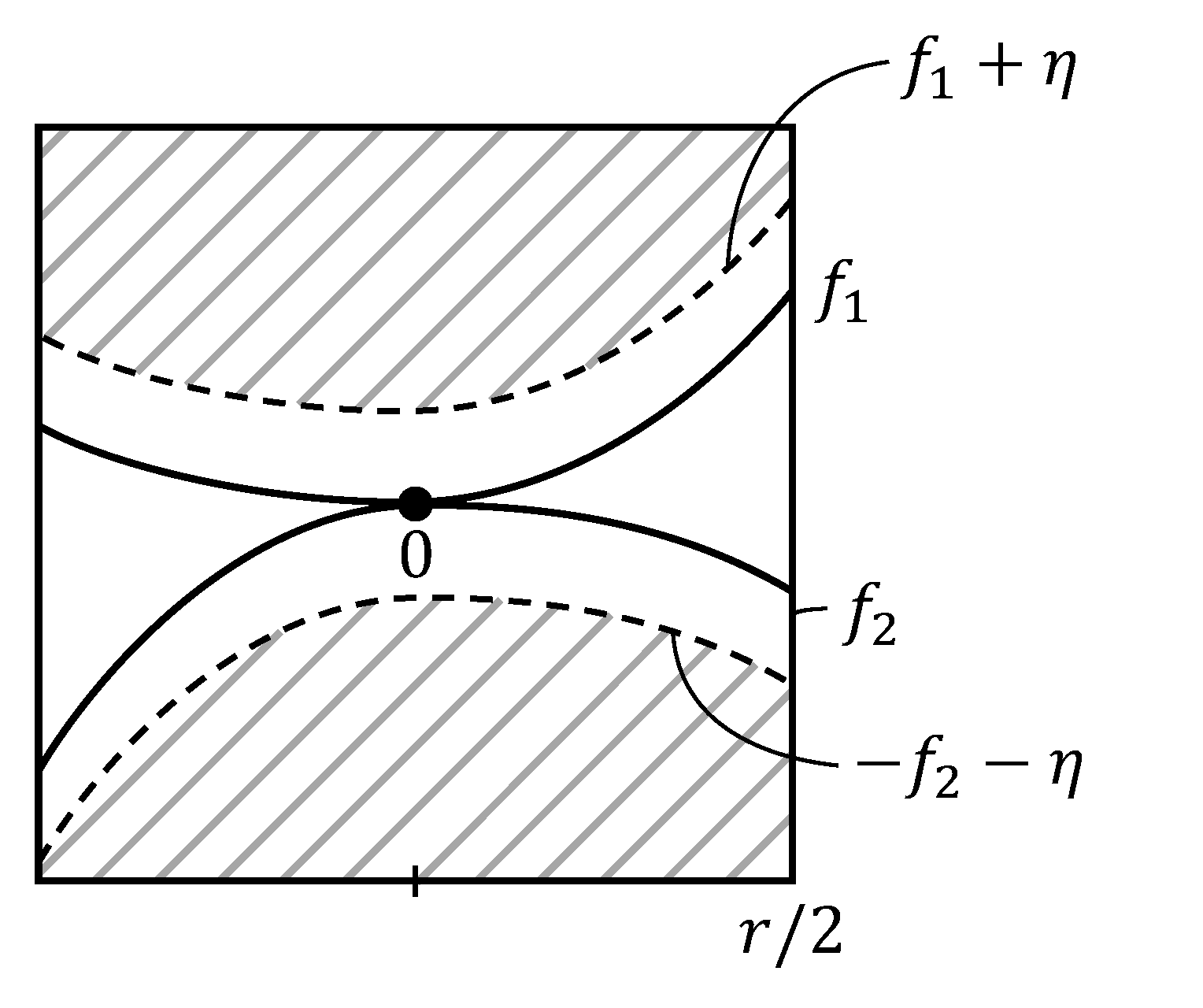}
    \caption{Initial profile}
    \label{FCo}
\end{figure}
This smallness only depends on $C^2$ regularity of $\hat{\Gamma}$.

Since
\[
    \partial_t f_{j,\eta}=-V\sqrt{1+|\nabla'f_{j,\eta}|^2}, \quad
    j=1,2,
\]
we observe that
\[
    \partial_t f_{1,\eta}(x',0)\le-1 \quad
    \partial_t f_{2,\eta}(x',0)=0
    \quad\text{for}\quad |x'|<r/2.
\]
Since $V$ is H\"older continuous,
\begin{align*}
    &V\le-\frac12 \quad\text{for}\quad j=1 \\
    &V\ge\frac{1}{\sqrt{1+\|\nabla'f_{2,\eta}\|_\infty^2}}\cdot\frac14 \quad\text{for}\quad j=2
\end{align*}
for $|x'|<r/2$ and $t\in(0,t_1]$ for sufficiently small $t_1\le T$ depending only on the H\"older regularity $A_1$.
 (Actually, H\"older norm in time is enough.)
 Here $\|\nabla'f_{2,\eta}\|_\infty$ is a $L^\infty$ norm in $B_r\times(0,T)$.
 We take $\delta_3$ smaller so that $\|\nabla'f_{2,\eta}\|_\infty$ is finite.
 This smallness depends only on $C^2$ regularity of $\hat{\Gamma}$.
 We thus conclude that
\begin{gather*}
    \partial_t f_{1,\eta}(x',t) \le -\frac12 \\
    \partial_t f_{2,\eta}(x',t) \ge \frac14
\end{gather*}
for $|x'|<r/2$, $t\in(0,t_1)$.
 This implies that
\begin{align*}
    f_{1,\eta}(0,t)+f_{2,\eta}(0,t)
    &= \int_0^t\left(\partial_t f_{1,\eta}(0,\tau)+\partial_t f_{2,\eta}(0,\tau)\right)\,d\tau+2\eta \\
    &\le \left(-\frac12+\frac14\right)t+2\eta
    = -\frac{t}{4}+2\eta.
\end{align*}
If $\eta_1=t_1/16$, then
\[
    f_{1,\eta}(0,t)+f_{2,\eta}(0,t)
\]
becomes negative for some $t\in(0,t_1)$, $\eta\in(0,\eta_1)$.
 This means that $\partial\Omega_\eta(t)$ should have self-intersection at some time $t>0$ provided that $\eta_1$ is small.
\end{proof}

We are now in position to state a detailed form of Theorem~\ref{TMain}.
\begin{thm} \label{TCol}
Assume that $2/p+d/q<1$, $p>2$, $q>d$.
 Let $\Omega$ be a $C^3$ splash domain $\mathbb{R}^d$ with compact boundary $\hat{\Gamma}$.
 Let $P\in\hat{\Gamma}$ be a point of self-intersection.
 For a given $\eta\in(0,\eta_0)$ let $\Omega_{0,\eta}$ be a domain $\Omega_{(\eta)}=\Omega^{h_{0\eta}}$ constructed in Section~\ref{SSC} with $h_{0\eta}\in B_{p,q}^{3-1/p-1/q}(\hat{\Gamma})$ such that $\|h_{0\eta}\|_{B_{q,p}^{3-1/p-1/q}}\le\varepsilon$, where $\varepsilon$ is given in Theorem~\ref{TLW}.
 Then there exists $\eta_1\in(0,\eta_0)$ and initial velocity field $u_{0\eta}\in B_{q,p}^{2(1-1/p)}$ satisfying the compatibility condition \eqref{EComp} such that $\partial\Omega_\eta(t)$ of the solution $\left(u_\eta,\varpi_\eta,\Omega_\eta(t)\right)$ (constructed in Theorem~\ref{TLW}) to \eqref{ENS}, \eqref{EIn} with $u|_{t=0}=u_{0\eta}$, $\Omega_\eta(t)|_{t=0}=\Omega_{0,\eta}$ has a self-intersection point near $P$ in a finite time provided that $\eta\in(0,\eta_1)$.
Moreover, the solution $u_\eta$ and the mean curvature $\kappa_\eta$ of $\partial\Omega_\eta(t)$ is bounded as $t$ tends to the collapsing time. 
\end{thm}
\begin{proof}
Since $\hat{\Gamma}$ is compact, it is of type $(r,\delta_0,K)$, we may take $K$ as small as we like by taking $r$ sufficiently small.
 By construction of $f_\eta$ and $h_{0\eta}$, for any $\varepsilon_1>0$, $\|h_{0\eta}\|_{C^3}\le\varepsilon_1$ by choosing $\delta_2$ small by Proposition~\ref{PCh}.
 This implies that $\|h_{0\eta}\|_{B_{q,p}^{3-1/p-1/q}}\le\varepsilon$ by choosing $\varepsilon_1$ smaller.

We shall construct initial velocity.
 We may take $r$ small so that the Jacobi matrix of coodinate change to normal coodinate is close to the identity in $\iota_{j,P}(C_r)$.
 We shall fix $r$.
 Let $b$ be a $C^2$ vector field on $\partial\Omega_{0,\eta}$ such that $b\equiv1$ on $\iota_{1,P}(C_{r/2})$ and $b\equiv0$ on $\partial\Omega_{0,\eta}\setminus\iota_{1,P}(C_r)$.
 By Proposition~\ref{PIV}, there is a vector field $u_{0\eta}\in W^{
 2,q}(\Omega_{0,\eta})$ such that
\[
    \|u_{0\eta}\|_{W^{2,q}}\le B
\]
with $B$ independent of $\eta$.
 (This constant $B$ depends only on $C^2$ regularity of $\hat{\Gamma}$.)
 Moreover, it satisfies the compatibility condition.

We solve the problem by Theorem~\ref{TLW} and found that
\begin{align*}
    &h_\eta\in L^p( 0,T;W^{3-1/q,q}(\hat{\Gamma}) )
    \cap W^{1,p}( 0,T;W^{2-1/q,q}(\hat{\Gamma}) )=:X_T \\
    &u_\eta^\# = u_\eta\circ\Xi\in L^p( 0,T;W^{2,q}(\Omega_0) )
    \cap W^{1,p}( 0,T; L^q(\Omega) )=:Y_T
\end{align*}
for some $T>0$ independent of $\eta$ and $u_{0,\eta}$.
 Moreover, Theorem~\ref{TLW} implies that $\|h_\eta\|_\infty\le\delta_2$, $\|\nabla_\Gamma h_\eta\|\le\delta_3$ by taking $\varepsilon$ smaller.
 By interpolation (cf.\ \cite{PS})
\[
    Y_T \subset H^{s,p}( 0,T;B_{q,p}^{2(1-s)}(\Omega) ),
    \quad 0<s<1,
\]
where $H^{s,p}$ is the space of Bessel potential.
 We take $s>1/p$ to get
\[
    H^{s,p}( 0,T;B_{q,p}^{2(1-s)}(\Omega) )
    \subset C^\nu( 0,T;B_{q,p}^{2(1-s)}(\Omega) )
\]
by Morrey's inequality with $\nu=s-1/p$.
 Since $2/p+d/q<1$,
\begin{align*}
    2(1-s)=2-2\left(\nu+\frac{1}{p}\right)
    &=1-2\nu+1-\frac{2}{p} \\
    &>1-2\nu+\frac{d}{q}.
\end{align*}
We take $\nu$ small so that $2\nu<d/q$.
 Then
\[
    B_{q,p}^{2(1-s)} \subset W^{1,q}.
\]
Again by Morrey's inequality (see e.g.\ \cite{Ev}), we see $W^{1,q}\subset C^{\nu_1}(\bar{\Omega})$ for $\nu_1=1-d/q$.
 Thus $Y_T\subset C^\mu\left(\Omega\times[0,T)\right)$ for $\mu=\min(\nu,\nu_1)$.
 We thus conclude that
\[
    \|V\|_{C^\mu} \le C\|u_\eta\|_{Y_T}.
\]
We are now able to apply Proposition~\ref{PCol} and we obtain a desired result of collapse.
We note that $u_\eta$ and $\kappa_\eta$ are bounded on $(0,T)$.
 Indeed, $h_\eta\in X_T$ and $u\in Y_T$ implies that
\[
    Lh_1\in C( [0,T);B_{q,p}^{3-2/p}(\Omega) ) \quad
    u_\eta\in C( [0,T);B_{q,p}^{2(1-1/p)}(\Omega) ).
\]
The boundedness of $u_\eta$ and $\kappa_\eta$ comes from the Sobolev embedding as discussed in Subsection~\ref{SSDiff}.
 The proof is now complete.
\end{proof}

\begin{proof}[Proof of Theorem~\ref{TMain}]
Theorem~\ref{TMain} is essentially contained in Theorem~\ref{TCol} except we did not mention that the monotone limit of $\Omega_{0,\eta}$ as $\eta\to0$ is a splash domain close to $\Omega$.
 Fortunately, this follows from the construction of $\Omega_{(\eta)}$. 
\end{proof}

%%%%%%%%%%%%%%%%%%%%%%%%%%%%%%%%%%%%%%%%%%%%%%%%%
\section{Sketch of the proof of the well-posedness} \label{SW} % Section 5

As noticed in Section~\ref{SSL}, the local well-posedness in Theorem~\ref{TLW} is essentially a special case of \cite[Theorem~6.1]{Sh} when the boundary of $\Omega$ is embedded.
 The proof of \cite[Theorem~6.1]{Sh} is very involved.
 We just briefly sketch the proof of Theorem~\ref{TLW} to clarify why his theory still works for a splash domain.

Let us first list the steps of the proof of Theorem~\ref{TLW}.
 Let $\Omega$ be a $C^3$ splash domain with compact boundary $\hat{\Gamma}$ in a domain $\Omega_\delta$ with $\delta$-wing constructed from $\Omega$.
 We postulate that a moving domain $\Omega(t)$ in $\Omega_\delta$ is of the form $\Omega^{h(t)}$ defined in Section~\ref{SSH}, where $h=h(y,t)$ is a function on $\hat{\Gamma}$ with value in $(-\delta_1,\delta_1)$.
\begin{enumerate}
\item [(i)] (Coordinate change by the Hanzawa transform)
 For $h$, we consider the Hanzawa transform $\Xi_{h(t)}$ and set
\[
    u_h^\#(y,t):=u\left(\Xi_{h(t)}(y),t\right),\quad
    \varpi_h^\#(y,t)=\varpi\left(\Xi_{h(t)}(y),t\right).
\]
We derive the equation for $u_h^\#$, $\varpi_h^\#$ and $h$ in a fixed splash domain $\Omega$.
 Note that the Hanzawa transform is given for a domain $\Omega^h$ in $\Omega_\delta$ (not in $\mathbb{R}^d$).
\item [(i\hspace{-0.1em}i)] (Linearization)
 We linearize the equation for $u_h^\#$, $\varpi_h^\#$ and $h$ around $h=0$.
\item [(i\hspace{-0.1em}i\hspace{-0.1em}i)] (Analysis of the linearized equation)
 We derive maximal regularity estimate for the linearized equation.
 Here we invoke the assumption that the surface tension coefficient $\sigma>0$.
\item [(i\hspace{-0.1em}v)] (Construction of a solution)
 We estimate nonlinear terms and construct a solution in a short time by Banach's fixed point theorem assuming that $h|_{t=0}$ is small.
 To handle nonlinear terms we invoke assumptions $2/p+d/q<1$, $p>2$, $q>d$
.
\end{enumerate}
In steps (i\hspace{-0.1em}i\hspace{-0.1em}i), (i\hspace{-0.1em}v), the problem is reduced on a chart $\iota_{j,P}(C_r)$ so the embeddedness of the boundary of $\Omega$ is unnecessary.

%%%%%%%%%%%%%%%%%%%%%%%%%%%%%%%%%%%%%%%%%%%%%%%%%
\subsection{Transformed problem} \label{SSCC} % Section 5.1

We write $x=\Xi_{h(t)}(y)$ and its inverse by $y=Z_{h(t)}(y)$.
 We simply write $u^\#$, $\varpi^\#$ instead of $u_h^\#$, $\varpi_h^\#$.

\vspace{0.7em}
\noindent{\bf Momentum and mass conservation laws.}
 We first transform the first two equations of \eqref{ENS}.
 We first observe that
\[
    (\nabla_x u)(x,t)
    =(\nabla_y u^\#)(y,t)\nabla_x y
    = (\nabla_y u^\#)(y,t)-F_1(u^\#,h)
\]
with
\[
    F_1(u^\#,h)
    :=(\nabla_y u^\#)(y,t)(I-\nabla_x y).
\]
Here we use the convention that $(\nabla_x w)_{ij}=(\partial w^i/\partial x_j)$ for a vector $w$.
 By this formula,
\[
    (\operatorname{div}_x u)(x,t)
    =(\operatorname{div}_y u^\#)(y,t)
    -\mathcal{D}(u^\#,h),
\]
where
\[
    \mathcal{D}(u^\#,h)
    :=\sum_{j=1}^d(\partial_{y_j}u^{\#j})(y,t)(1-\partial_{x_j}y_j)
    -\sum_{\substack{j,m=1\\j\neq m}}^d(\partial_{y_m}u^{\#j})(y,t)\partial_{x_j}y_m.
\]
Similarly, for a $d\times d$ matrix $M^\#(y)$, we define the vector $\mathcal{D}(M^\#,h)$ whose $k$th component $\mathcal{D}(M^\#,h)_k$ is set to be
\[
    \mathcal{D}(M^\#,h)_k
    :=\sum_{j=1}^d \partial_{y_j}M_{j_k}^\#(1-\partial_{x_j}y_j)
    -\sum_{j=1}^d \sum_{\ell\neq j} \partial_{y_\ell}M_{j_k}^\# \partial_{x_j}y_\ell.
\]
Then
\[
    \operatorname{div}_x M
    =\operatorname{div}_y M^\#- \mathcal{D}(M^\#,h),
\]
where $M(x)=M^\#\left(Z_h(x)\right)$.
 Thus,
\[
    \operatorname{div}_x \mathbb{S}(u,\varpi)
    =\operatorname{div}_y \mathbb{S}(u^\#,\varpi^\#)+ \mathcal{D}(\varpi^\#I,h)+F_2(u^\#,h)
\]
with
\begin{align*}
   F_2(u^\#,h):=
   &-2\mu\mathcal{D}\left(\mathbb{D}_y(u^\#),h\right)
   -\mu\operatorname{div}_y F_1(u^\#,h)
   -\mu\operatorname{div}_y F_1(u^\#,h)^T \\
   &+\mu\mathcal{D}\left(F_1(u^\#,h),h\right)
   +\mu\mathcal{D}\left(F_1(u^\#,h)^T,h\right).    
\end{align*}
Since
\[
    (\partial_t u)(x,t)
    =(\partial_t u^\#)(y,t)
    +\nabla_y u^\#(y,t)\partial_t Z_h,
\]
we observe that the equation
\[
    \partial_t u+u\cdot\nabla_x u
    -\operatorname{div}\mathbb{S}(u,\varpi)=0
\]
is transformed into
\[
    \partial_t u^\#
    -\operatorname{div}\mathbb{S}(u^\#,\varpi^\#)
    =F_2(u^\#,h)+\mathcal{D}(\varpi^\#,h)-\nabla_y u^\#(y,t)(\partial_t Z_h)
    -u\cdot\nabla_xu.
\]
For later convenience, we eliminate $\varpi^\#$ in the right-hand side.
 Since $\nabla\varpi=-\partial_tu-u\cdot\nabla_xu+2\mu\operatorname{div}_x\mathbb{D}(u)$, we see
\[
    \nabla_y\varpi^\# = G(u^\#,h)\cdot \nabla_y x
\]
with
\begin{align*}
        G(u^\#,h)
        &=(-\partial_t u^\#-\nabla_y u^\# \cdot\partial_t Z_h)
        +\left(-(u^\#\cdot\nabla_y)u^\#+u^\#\cdot F_1(u^\#,h)\right)\\
        &+\mu\left(\operatorname{div}_y(\nabla_yu^\#\cdot\nabla_xy)
        -\mathcal{D}(\nabla_yu^\#\cdot\nabla_xy,h)\right)\\
        &+\mu\left(\operatorname{div}_y(\nabla_yu^\#\cdot\nabla_xy)
        -\mathcal{D}\left((\nabla_yu^\#\cdot\nabla_xy)^T,h\right)\right).
\end{align*}
Thus
\begin{align*}
        \mathcal{D}(\varpi^\#I,h)
        &=\nabla_y\varpi^\#(I-\nabla_xy) \\
        &=G(u^\#,h)\cdot(\nabla_yx-I)
        =: F_3(u^\#,h)
\end{align*}
since $\nabla_yx\cdot\nabla_xy=I$.
 We now observe that
\[
    \partial_t u+(u\cdot\nabla)u
    - \operatorname{div}\mathbb{S}(u,\varpi)=0
\]
is transformed into
\begin{equation} \label{ET1}
    \partial_t u^\# - \operatorname{div}\mathbb{S}(u^\#,\varpi^\#)
    =\mathcal{F}(u^\#,h)
\end{equation}
with $\mathcal{F}(u^\#,h):=F_2(u^\#,h)+F_3(u^\#,h)-(u^\#\cdot\nabla_y)u^\#\cdot\nabla_xy-\nabla_yu^\#\partial_tZ_h$.
 The equation $\operatorname{div}u=0$ is easily transformed into
\begin{equation} \label{ET2}
    \operatorname{div}u^\#
    =\mathcal{D}(u^\#,h).
\end{equation}

\vspace{0.7em}
\noindent{\bf Kinematic condition.}
 We shall write the normal $\mathbf{n}_{h(t)}$ and the normal velocity $V_{\hat{\Gamma}_h}$ of $\hat{\Gamma}_{h(t)}$ as a function on $\hat{\Gamma}$.
 We just recall \cite[Chapter~2]{PS}.
 The normal is of the form
\begin{equation} \label{ENor}
\mathbf{n}_h=\mathcal{B}_h(\mathbf{n} - \mathcal{A}_h)
\end{equation}
with
\begin{align} \label{Ah:Bh}
\mathcal{A}_h := (I-h\mathcal{L})^{-1}\nabla_{\hat{\Gamma}}h, \quad \mathcal{B}_h := \left(1+|\mathcal{A}_h|^2\right)^{-1/2},
\end{align}
where $\mathbf{n}$ denotes the normal of $\hat{\Gamma}$ and $\mathcal{L}$ denotes the Weingarten tensor of $\hat{\Gamma}$; see \cite[(2.44), (2.45)]{PS}.
 Here $\nabla_{\hat{\Gamma}}$ denotes the surface gradient defined by $\nabla_{\hat{\Gamma}}=(I-\mathbf{n}\otimes\mathbf{n})\nabla$.
 The normal velocity $V_{\hat{\Gamma}_h}$ of $\hat{\Gamma}_{h(t)}$ is defined by
\[
    V_{\hat{\Gamma}_h}(y,t)
    =\partial_t\Xi_{h(t)}(y)\cdot\mathbf{n}_{h(t)}(y), \quad
    y\in\hat{\Gamma};
\]
see \cite[Chapter~1, \S3.2 (a)]{PS}.
 Since the Weingarten tensor $\mathcal{L}$ leaves the tangent space $T_{\hat{P}}\hat{\Gamma}$ at a point $\hat{P}\in\hat{\Gamma}$ invariant, we see, by \eqref{ENor},
\[
    V_{\hat{\Gamma}_h}(y,t)=\mathcal{B}_{h(t)}(y)\partial_t h(y,t).
\]
The kinematic condition $V=u\cdot\mathbf{n}$ on $\hat{\Gamma}_h$ can be written as
\[
    \mathcal{B}_{h(t)}(y)\partial_th(y,t)
    =\mathcal{B}_{h(t)}(y)\left(u^\#(y,t)\cdot\mathbf{n}(y)-u^\#(y,t)\cdot\mathcal{A}_{h(i)}(y)\right), \quad
    y\in\hat{\Gamma}.
\]
Rearranging this relation, we obtain
\begin{equation} \label{EKi}
    \partial_t h+u^\#\cdot\nabla_\Gamma h - u^\#\cdot\mathbf{n}=\mathcal{E}(u^\#,h)
\end{equation}
with
\[
    \mathcal{E}(u^\#,h):=-u^\#(\mathcal{A}_{h(t)}-\nabla_\Gamma h).
\]
It turns out it is more convenient to put $u^\#\cdot\nabla_\Gamma h$ in the left since $\mathcal{A}_{h(t)}-\nabla_\Gamma h$ is expected to be small.

\vspace{0.7em}
\noindent{\bf Balance by surface tension.}
 We next write the force balance
\begin{equation} \label{EFB}
    \mathbb{S}(u,\varpi)\mathbf{n}_h
    = \sigma\kappa_h\mathbf{n}_h
\end{equation}
on the boundary $\hat{\Gamma}_h$.
 Since there is no time derivative, we shall fix $t$ and write $h(t)$ simply by $h$.
 Let $\mathbb{P}_h$ denote the tangential projection to $\hat{\Gamma}_h$, i.e.,
 \[
    \mathbb{P}_h:=I-\mathbf{n}_h\otimes\mathbf{n}_h.
 \]
 The equation \eqref{EFB} can be rewritten as
\begin{align}
    \mathbb{P}_h\left(\mathbb{D}(u)\mathbf{n}_h\right) &=0, \label{EFT} \\
    \mathbf{n}_h \mathbb{S}(u,\varpi)\mathbf{n}_h
    &= \sigma\kappa_h. \label{EFN}
\end{align}
By \eqref{ENor}, we see that
\[
    \mathbb{P}_h=\mathbb{P}_0+\mathcal{T}_h, \quad
    \mathbb{P}_0=I-\mathbf{n}\otimes\mathbf{n}
\]
with
\[
    \mathcal{T}_h:=-(\mathcal{B}_h^2-1)\mathbf{n}\otimes\mathbf{n}+\mathcal{B}_h^2
    (\mathcal{A}_h\otimes\mathbf{n}+\mathbf{n}\otimes\mathcal{A}_h-\mathcal{A}_h\otimes\mathcal{A}_h).
\]
By the coordinate change, \eqref{EFT} becomes
\begin{equation} \label{EFTT}
    \mathbb{P}_0 2\mu_0\mathbb{D}_y(u^\#)\mathbf{n}
    =\mathcal{J}_{\tan}(u^\#,h)
\end{equation}
with
\begin{align*}
    \mathcal{J}_{\tan}(u^\#,t) &:= \mathbb{P}_0 2\mu\mathbb{D}_y(u^\#)\mathcal{A}_h \\
    &+\mathbb{P}_0 \mu\left(F_1(u^\#,h)+F_1(u^\#,t)^T\right)\cdot(\mathbf{n}-\mathcal{A}_h) \\
    &+\mathcal{J}_h F_4(u^\#,h)\cdot(\mathbf{n}-\mathcal{A}_h) \\
    F_4(u^\#,h) &:= 2\mu\mathbb{D}_y(u^\#)-\mu F_1(u^\#,h)-\mu F_1(u^\#,h)^T.
\end{align*}
Let us write \eqref{EFN} in our coordinate system.
 Since the Weingarten tensor $\mathcal{L}$ leaves the tangent space $T_P\hat{\Gamma}$ invariant \cite[Chapter~2, \S1.2]{PS}, we have $\mathcal{A}_h\in T_P\hat{\Gamma}$.
 As a result
\[
    \kappa_h=\kappa +\Delta_{\hat{\Gamma}}h+H_1(h),
\]
where
\[
    H_1(h):=(\mathcal{B}_h-1)\kappa+(B_h-1)\Delta_{\hat{\Gamma}}h
    +\mathcal{B}_h(\mathbf{n}-\mathcal{A}_h)C(h).
\]
Here $C(h)$ is a linear combination of $h$ up to its first derivative whose coefficients are bounded and its bound only depends on $C^2$ regularity of $\hat{\Gamma}$.
 (It also includes a zero-th order term which is a given function and its size is also controlled by $C^2$-regularity of $\hat{\Gamma}$).
 The calculation is involved but it is the same as in \cite[\S3.3]{Sh}.
 A more systematic calculation of the mean curvature of a surface given by a height function is given in \cite[Chapter~2, \S2.5, (2.49)]{PS}.
 We do not use detailed form of $C(h)$.

Since
\[
    \mathbf{n}_h 2\mu\mathbb{D}_x(u)\mathbf{n}_h
    =\mathbf{n}2\mu\mathbb{D}_y(u^\#)\mathbf{n}
    -F_5(u^\#,h)
\]
with
\begin{align*}
    F_5(u^\#,h) &:=\mathbf{n}\left(\mu F_1(u^\#,h)+\mu F_1(u^\#,h)^T\right)\mathbf{n}
    -(\mathcal{B}_h^2-1) F_4(u^\#,h)\mathbf{n} \\
    &+\mathcal{B}_h^2 \left\{ \mathbf{n} F_4(u^\#,h)\mathcal{A}_h
    +\mathcal{A}_h F_4(u^\#,h)\mathbf{n}-\mathcal{A}_h F_4(u^\#,h)\mathcal{A}_h\right\},
\end{align*}
we now conclude that \eqref{EFN} can be written as
\begin{equation} \label{EFNT}
    \mathbf{n} 2\mu\mathbb{D}_y(u^\#)\mathbf{n}
    -\varpi^\#-\sigma\kappa-\sigma\Delta_{\hat{\Gamma}}h
    =\mathcal{J}_{\operatorname{nor}}(u^\#,h)
\end{equation}
with
\[
    \mathcal{J}_{\operatorname{nor}}(u^\#,h):=F_5(u^\#,h)+\sigma H_1(h).
\]

We now observe that our equation \eqref{ENS} is transformed into \eqref{ET1}, \eqref{ET2}, \eqref{EKi}, \eqref{EFTT}, \eqref{EFNT}, i.e.,
\begin{equation}
\begin{aligned} \label{ETFul}
    \partial_t u^\#-\operatorname{div}_y\mathbb{S}_y(u^\#,\varpi^\#)
    &=\mathcal{F}(u^\#,h) &&\text{in}\quad \Omega \\
    \operatorname{div}_y u^\#
    &=\mathcal{D}(u^\#,h) &&\text{in}\quad \Omega \\
    \partial_t h+u^\#\cdot\nabla_{\hat{\Gamma}}h-u^\#\cdot\mathbf{n}
    &=\mathcal{E}(u^\#,h) &&\text{on}\quad \hat{\Gamma} \\
    \mathbb{P}_0 2\mu\mathbb{D}_y(u^\#)\mathbf{n}
    &=\mathcal{J}_{\tan}(u^\#,h) &&\text{on}\quad \hat{\Gamma} \\
    \mathbf{n}2\mu\mathbb{D}_y(u^\#)\mathbf{n}-\varpi^\# -\sigma\kappa-\sigma\Delta_{\hat{\Gamma}}h
    &=\mathcal{J}_{\operatorname{nor}}(u^\#,h) &&\text{on}\quad \hat{\Gamma}.
\end{aligned}
\end{equation}
The last two equation \eqref{EFTT}, \eqref{EFNT} can be written in
\begin{equation} \label{EFBT}
    \mathbb{S}(u^\#,\varpi)\mathbf{n}-\sigma\mathbf{n}\Delta_{\hat{\Gamma}}h
    =\mathcal{J}+\sigma\mathbf{n}\kappa
\end{equation}
with $\mathcal{J}=\mathcal{J}_{\tan}+\mathbf{n}\mathcal{J}_{\operatorname{nor}}$.
 Note that the term $\sigma\mathbf{n}\kappa$ is determined if we fix $\hat{\Gamma}$.

It is convenient to prepare estimates of a kind of lift of $\mathcal{A}_h$ defined by
\begin{equation} \label{ELiftA}
    \hat{\mathcal{A}}_h(y)
    = \chi\left(\frac{d(y)}{\delta_1}\right)
    \left(I-(Lh)\mathcal{L}\right)^{-1}
    \nabla_{\hat{\Gamma}} (Lh)\left(\Pi(y)\right), \quad
    y\in\Omega.
\end{equation}
\begin{prop} \label{PDifA}
Assume the same hypotheses of Theorem~\ref{TInv} concerning $\Omega$.
 Assume that $h_i$ ($i=1,2$) satisfies assumptions of $h$ in Theorem~\ref{TInv} with constants $\varepsilon_0$ and $\varepsilon_1$.
 Then by taking $\varepsilon_1$ smaller,
\[
    \left|\hat{\mathcal{A}}_{h_1}(y)-\hat{\mathcal{A}}_{h_2}(y)\right|
    \le C_2 \left(\left|L(h_1-h_2)(y)\right|
    +\left|\nabla L(h_1-h_2)(y)\right|\right), \quad
    y\in\Omega
\]
with some constant $C_2$ depending only on $\varepsilon_0$, $\varepsilon_1$, $\delta_1$ and $\Omega$ through $(r,\delta_0,K)$.
 In particular,
\[
    \left|\hat{\mathcal{A}}_{h_i}(y)\right|
    \le C_2 \left(\left|Lh_i(y)\right|
    +\left|\nabla Lh_i(y)\right|\right), \quad
    i=1,2.
\]
\end{prop}

This is easy to prove.
 We note that we have to take $\varepsilon_1$ smaller so that $I-(Lh)\mathcal{L}$ is invertible.
 The second assertion follows from the first estimate by taking one of $h_i$ equals zero.

%%%%%%%%%%%%%%%%%%%%%%%%%%%%%%%%%%%%%%%%%%%%%%%%%
\subsection{Linearization and approximation} \label{SSLP} % Section 5.2

We would like to solve the initial value problem for \eqref{ETFul} locally-in-time only smallness assumption on $h|_{t=0}=h_0$ without assuming that initial velocity is small.

We approximate initial data $u_0^\#$ as
\[
    u_\zeta = \frac1\zeta \int_0^\zeta T(s)u_0^\#\, ds, \quad \zeta>0,
\]
where $T(s)$ is an analytic semigroup in $L^q$ satisfying the maximal regularity property, i.e.,
\[
    \left\lVert T(t)u_0\right\rVert_{L^p\left(0,\infty;W^{2,q}(\Omega)\right)}
    +\left\lVert T(t)u_0\right\rVert_{W^{1,p}\left(0,\infty;L^q(\Omega)\right)}
    \le C\lVert u_0\rVert_{B_{q,p}^{2(1-1/p)}}.
\]
Moreover, it satisfies
\[
    \left\lVert T(t)u_0\right\rVert_{L^\infty\left(0,T;B_{p,q}^{2(1-1/p)}\right)}
    \le C\lVert u_0\rVert_{B_{q,p}^{2(1-1/p)}}.
\]
For example, we take $T(s)=e^{s(\Delta-1)}$, where $\Delta$ denotes the Dirichlet Laplacian.
 We linearize \eqref{EKi} around $u_\zeta$ and observe that
\begin{equation} \label{Eki2}
    \partial_t h+(u_\zeta\cdot\nabla_\Gamma)h-u^\#\cdot\mathbf{n}
    =\mathcal{E}+(u_\zeta-u^\#)\nabla_{\hat{\Gamma}}h.
\end{equation}
Note that if we linearize around $u_0^\#$, the regularity of coefficient not enough for analysis of \eqref{EKi}.
 We consider the system \eqref{ET1}, \eqref{ET2}, \eqref{Eki2}, \eqref{EFBT}, i.e., 
\begin{equation}
\begin{aligned} \label{ETFul2}
    \partial_t u^\#-\operatorname{div}_s\mathbb{S}_y(u^\#,\varpi^\#)
    &=\mathcal{F} &&\text{in}\quad \Omega \\
    \operatorname{div}_y u^\#
    &=\mathcal{D} &&\text{in}\quad \Omega \\
    \partial_t h+u_\zeta\cdot\nabla_\Gamma h-u^\#\cdot\mathbf{n}
    &=\mathcal{E}+(u_\zeta-u^\#)\cdot\nabla_\Gamma h &&\text{on}\quad \hat{\Gamma} \\
    \mathbb{S}_y(u^\#,\varpi^\#)\mathbf{n}-\sigma\mathbf{n}\Delta_\Gamma h
    &=\mathcal{J}+\sigma\mathbf{n}\kappa &&\text{on}\quad \hat{\Gamma}
\end{aligned}
\end{equation}
under initial condition
\[
    u^\#|_{t=0}=u_0^\#, \quad
    h|_{t=0}=h_0.
\]
Here $\mathcal{J}$ is a vector so that $(\mathcal{J})_{\tan}=\mathcal{J}_{\tan}$, $\mathcal{J}\cdot\mathbf{n}-\mathcal{J}_{\operatorname{nor}}$.
 Since we are just interested in the existence of a local-in-time solution, lower order linear terms are included in the right-hand side.
 This is different from what is presented in \cite{Sh}, where global-in-time problem is also considered.

%%%%%%%%%%%%%%%%%%%%%%%%%%%%%%%%%%%%%%%%%%%%%%%%%
\subsection{Analysis of liearized problem} \label{SSMax} % Section 5.3

We consider a linear system and state an $L^p$-$L^q$ maximal regularity result for a system for a $C^3$ splash domain with compact boundary.
 This is essentially contained in results of \cite{Sh}, where he treated a general uniformly $C^3$ domain but whose boundary is embedded in $\mathbb{R}^d$.

We consider a generalized Stokes system coupled with a free surface $h$ of the form
\begin{equation}
\begin{aligned} \label{EGS}
    \partial_t u-\operatorname{div}\mathbb{S}(u,\varpi)
    &=f &&\text{in}\quad \Omega \\
    \operatorname{div} u
    &=\operatorname{div} g &&\text{in}\quad \Omega \\
    \partial_t h+U_\zeta\cdot\nabla_\Gamma h-u\cdot\mathbf{n}+L_1(h)
    &=\theta &&\text{on}\quad \hat{\Gamma} \\
    \mathbb{S}(u,\varpi)\mathbf{n}-\left(L_2(h)+\sigma\Delta_\Gamma h\right)\mathbf{n}
    &=\omega &&\text{on}\quad \hat{\Gamma} \\
    \left.(u,h)\right|_{t=0}
    &=(u_0,h_0) &&\text{in}\quad \Omega\times\hat{\Gamma}.
\end{aligned}
\end{equation}
For boundary data $U_\zeta$ we assume that there exist positive constants $M_0$, $a_0$ ($<1$), $b_0$ such that
\begin{equation} \label{EAU}
	\left\{
\begin{array}{l}
     \left|U_\zeta(x)\right|\le M_0,\ 
     \left|U_\zeta(x)-U_\zeta(y)\right|\le M_0,\ x,y\in\hat{\Gamma} \\
     \|U_\zeta\|_{W^{2-1/s,s}(\hat{\Gamma})}\le M_0\zeta^{-b_0}\ 
     \text{for all}\ \zeta\in(0,1)\ \text{for some}\ s\in(d,\infty).
\end{array}
\right.
\end{equation}
The operator $L_1$ and $L_2$ are linear operators with bounds $M_1$ and $M_2$ satisfying
\[
    \left\|L_1(v)\right\|_{W^{2-1/q,q}(\hat{\Gamma})}
    \le M_1\|v\|_{W^{1,q}(\Omega)}, \quad
    \left\|L_2(h)\right\|_{W^{1, q}(\Omega)}
    \le M_2\|h\|_{W^{2,q}(\Omega)};
\]
when we consider $h$ in $\Omega$, we consider its lift $Lh$ in Section~\ref{SH}.

It is convenient to introduce a necessary solenoidal space
\[
    J_q(\Omega):= \left\{ f\in L^q(\Omega) \biggm|
    \int_\Omega f\cdot\nabla\phi\,dx=0
    \ \text{for all}\ \phi\in\hat{W}_0^{1,q'}(\Omega) \right\}
\]
where $q'=q/(q-1)$ represents the conjugate exponent of $q$.
 The next theorem is a trivial modification of the maximal $L^p$-$L^q$ regularity for \eqref{EGS} obtained by \cite[Corollary~4.6]{Sh}.
\begin{thm}[\cite{Sh}] \label{TShMR}
Let $\Omega$ be a $C^3$ splash domain in $\mathbb{R}^d$ with compact boundary $\hat{\Gamma}$ of type $(r,\delta_0,K)$.
 Let $1<p,q<\infty$ with $2/p+1/q\neq1$.
 Let
\[
    (u_0,h_0)\in B_{q,p}^{2(1-1/p)}(\Omega)\times B_{q,p}^{3-1/p-1/q}(\hat{\Gamma})
\]
be an initial data of \eqref{EGS}.
 Then there exists a constant
\[
    \gamma_*= \gamma_*(\mu,\sigma,p,q,r,\delta_0,K,a_0,b_0,M_0,M_1,M_2)>0
\]
such that for any $\gamma\ge\gamma_*$ and
\[
	\left\{
\begin{array}{l}
     f\in L^p\left(0,T;L^q(\Omega)\right),\ 
     \theta\in L^p( 0,T;W^{2-1/q,q}(\hat{\Gamma}) ), \\
     e^{-\gamma t}\operatorname{div}g\in L^p(\mathbb{R};W^{1,q}(\Omega) )
     \cap H^{1/2,p}(\mathbb{R};L^q(\Omega) ), \\
     e^{-\gamma t} g \in W^{1, p}(\mathbb{R};L^q(\Omega) ), \\
     e^{-\gamma t}\omega\in L^p(\mathbb{R};W^{1,q}(\Omega) )
     \cap H^{1/2,p}(\mathbb{R};L^q(\Omega) ) 
     \cap C(\mathbb{R}, B_{q,p}^{1-2/p} (\Omega) )\\
     \text{(where the lift of}\ \omega\ \text{from}\ \hat{\Gamma}\ \text{to}\ \Omega\ \text{is also denoted by}\ \omega\text{)},
\end{array}
\right.
\]
that satisfies compatibility conditions
\begin{gather*}
    u_0-g|_{t=0}\in J_q(\Omega), \\
    \left(\mu\mathbb{D}(u_0)\mathbf{n}\right)_{\tan}
    =\left(\omega|_{t=0}\right)_{\tan}
    \in B_{q,p}^{1-\frac2p-\frac1q}(\Omega)
    \quad\text{if}\quad 2/p+1/q<1,
\end{gather*}
the generalized Stokes problem \eqref{EGS} admits a unique solution $(u,\varpi,h)$ with
\begin{align*}
    & u\in L^p(0,T; W^{2,q}(\Omega) )
    \cap W^{1,p}( 0,T; L^q(\Omega) ) \\
    & \varpi\in L^p( 0,T; W^{1,q}(\Omega)
    + \hat{W}_0^{1,q}(\Omega) ) \\
    & h\in L^p( 0,T; W^{3-1/q,q}(\hat{\Gamma}) )
    \cap W^{1,p}( 0,T; W^{2-1/q,q}(\hat{\Gamma}) )
\end{align*}
that satisfies the estimate
\begin{align*}
&\|u\|_{L^p(0,T;W^{2,q}(\Omega))} + \|\partial_t u\|_{L^p(0,T;L^q (\Omega))} + \|h\|_{L^p(0,T;W^{3-1/q,q}(\hat{\Gamma}))} + \|\partial_t h\|_{L^p( 0,T;W^{2-1/q,q}(\hat{\Gamma}) )} \\
&\ \ \le Ce^{2\gamma\zeta^{- b_0} T} \Bigl\{ \|u_0\|_{B_{q,p}^{2(1-1/p)}(\Omega)}
    + \kappa^{-b_0} \|h_0\|_{B_{q,p}^{3-1/p-1/q}(\hat{\Gamma})} \\
&\ \ \quad + \|f\|_{L^p( 0,T;L^q(\Omega) )} + \|e^{-\gamma t}\partial_t g\|_{L^p(\mathbb{R}; L^q(\Omega))} + \|e^{-\gamma t}(\operatorname{div}g,\omega)\|_{L^p( \mathbb{R}; W^{1,q}(\Omega) )} \\
&\ \ \quad + \|e^{-\gamma t}(\operatorname{div}g,\omega)\|_{H^{1/2,p}(\mathbb{R};L^q(\Omega))}
    + \|\theta\|_{L^p( 0,T;W^{2-1/q,q}(\hat{\Gamma}) )} \Bigr\}
\end{align*}
with some $C=C(\mu,\sigma,p,q,r,\delta_0,K,M_0)>0$ independent of $\zeta\in(0,1)$.
\end{thm}
\begin{rem} \label{RBes}
Since this theorem is essentially taken from \cite[Corollary~4.6]{Sh} which corresponds to the case where $\partial\Omega$ is embedded, in this article it is sufficient for us to clarify the dependence of constants without presenting the whole detailed proof.
 Moreover, since $u$ must belong to $C( [0,T);B_{q,p}^{2(1-1/p)}(\Omega) )$ by $u\in Y_T$, $\mathbb{D}(u)$ is continuous from $[0,T)$ to $B_{q,p}^{1-2/p}(\Omega)$.
 This means that $\omega\in C( [0,T);B_{q,p}^{1-2/p}(\Omega) )$.
 The compatibility condition for $\omega$ is necessary when $1-2/p-1/q>0$ because the spatial trace exists for this exponent.
 We thus add $\omega\in C( [0,\infty);B_{q,p}^{1-2/p}(\Omega) )$.
\end{rem}
\begin{proof}[Brief sketch of the proof]
Actually, the proof is very involved.
 To show the maximal regularity, we prove $\mathcal{R}$-sectoriality by considering its resolvent problem.
 The detailed proof is found in \cite[Section~5]{Sh}.

The problem is easily reduced to the case without $L_1$ and $L_2$ (cf.\ \cite[\S5.9]{Sh}).
 By considering a suitable partition of unity (cf.\ \cite[Proposition~2.2]{Sh}) the resolvent problem corresponding to \eqref{EGS} is reduced to the problem in whole space and in a slightly perturbed $C^3$ half-space.
 For a splash domain we consider a partition of unity in a neighborhood not of the closure of $\Omega$ in $\mathbb{R}^d$ but of the closure of $\Omega$ in $\Omega_\delta$ i.e., $\Omega$ with $\delta$-wing.
 Once we notice this fact, all remaining procedures are the same as \cite[\S5.6--5.9]{Sh}.

Let us briefly recall a problem for a perturbed half-space $\mathbb{R}_f^d=\left\{(x',x_d)\in\mathbb{R}^d; f(x')<x_d\right\}$.
 By a suitable diffeomorphism from the half-space $\mathbb{R}_+^d$ to $\mathbb{R}_f^d$, the resolvent problem is reduced to a perturbation of the problem in a half-space, i.e.,
\begin{equation} \label{ERe1}
	\left\{
\begin{aligned}
    \lambda u - \operatorname{div}\mathbb{S}(u,\varpi)
    &=f_0 &&\quad\text{in}\quad \mathbb{R}_+^d \\
    \operatorname{div} u
    &=\operatorname{div}g_0 &&\quad\text{in}\quad \mathbb{R}_+^d \\
    \lambda^h+U_\zeta\cdot\nabla'h-u\cdot\mathbf{n}
    &=\theta_0 &&\quad\text{on}\quad \partial\mathbb{R}_+^d \\
    \left(\mathbb{S}(u,\varpi)-\sigma\Delta'h\right)\mathbf{n}
    &=\omega_0 &&\quad\text{on}\quad \partial\mathbb{R}_+^d,
\end{aligned}
\right.
\end{equation}
where $\lambda$ is a complex parameter.
 We may assume $U_\zeta$ is a constant by localization.
 In this procedure, we see $\gamma_*$ depends on the regularity of $\hat{\Gamma}$.
 Let us concentrate \eqref{ERe1}.
 When $(f_0,\operatorname{div}g_0,\omega_0)=0$, the resolvent problem \eqref{ERe1} can be transformed into an ODE system by applying the partial Fourier transform to the tangential direction.
 The explicit formula of the Fourier transform $(\hat{u},\hat{\varpi},\hat{h})$ is obtained (cf.\ \cite[Theorem~5.13]{Sh}).
 In this case, $\mathcal{R}$-boundedness of the resolvent is not diffcult by checking the Fourier multiplier of the solution operator.
 The dependence of $r_*$ on $\sigma$, $\mu$, $M_0$ is due to the estimate in the Fourier multipliers (cf.\ \cite[Lemma~5.14]{Sh}).
 Note that $U_\zeta$ is assumed to be a constant by localization and perturbation.
 If $\sigma=0$, then we can still write a solution explicitly but there is no regularizing effect for $h$ so we do not get maximal regularity estimates.

The general case of \eqref{EGS} with $(f_0,\operatorname{div}g_0,\omega_0)\neq0$ can be reduced to the case where $(f_0,\operatorname{div}g_0,\omega)=0$ by solving the resolvent problem of the Stokes equations:
\[
	\left\{
\begin{aligned}
    \lambda u - \operatorname{div}\mathbb{S}(u,\varpi)
    &=f_* &&\text{in}\quad \mathbb{R}_+^d \\
    \operatorname{div} u
    &=\operatorname{div}g_* &&\text{in}\quad \mathbb{R}_+^d \\
    \mathbb{S}(u,\varpi)
    &=\omega_* &&\text{on}\quad \mathbb{R}^d
\end{aligned}
\right.
\]
(cf.\ \cite[\S5.3]{Sh}).
 To kill the pressure term we usual use some Helmholtz decompositions which is obtained by solving the weak Dirichlet problem mentioned in Section~\ref{SSL}.
 This is why we mention the solvability of a weak Dirichlet problem in Section~\ref{SSL}.
 The full resolvent problem including $L_1$, $L_2$ can be solved by considering a perturbation problem for the resolvent problem \eqref{ERe1} (cf.\ \cite[\S5.9]{Sh}).
 This is the reason why $\gamma_*$ deoends on $M_1$, $M_2$.
 Note that these perturbations are in some sense lower order perturbations.
 The dependence $C$ on $\mu$, $\sigma$, $M_0$ is due to the fact that the $\mathcal{R}$-bound for the solution operator of \eqref{ERe1}.
 The reason that $C$ is independent of $M_1$, $M_2$ is that the perturbation by $L_1$ and $L_2$ is squeezed in $\gamma_*$ by taking $\gamma_*$ larger.

We have thus constructed a solution with necessary estimates.
 As noted in \cite[Remark~4.5]{Sh}, the uniqueness of a solution follows by duality if $U_\zeta=0$.
 However, in the case of $U_\zeta\neq0$, the duality argument does not seem to work and one has to derive a priori estimates which need further assumptions on $\Omega$ such that inside $\Omega$ has a finite covering \cite[Definition~2.3]{Sh} which is fulfilled for an exterior or a bounded domain.
\end{proof}

%%%%%%%%%%%%%%%%%%%%%%%%%%%%%%%%%%%%%%%%%%%%%%%%%
\subsection{Remark on local solvability} \label{SSLS} % Section 5.4

We should solve \eqref{ETFul2} by using Theorem~\ref{TShMR}.
 The function spaces where constructed solution lies in $X_T$ for $h$ and $Y_T$ for $u$ and $Z_T$ for $\varpi$ as in Section~\ref{SSL}.
 We consider
\begin{equation}
\begin{aligned} \label{EMap}
    \partial_t u-\operatorname{div}_y\mathbb{S}_y(u,\varpi)
    &=\mathcal{F}(v,k) &&\text{in}\quad \Omega \\
    \operatorname{div}_y u
    &=\mathcal{D}(v,k) &&\text{in}\quad \Omega \\
    \partial_t h+u_\zeta\cdot\nabla_\Gamma h-u\cdot\mathbf{n}
    &=\mathcal{E}(v,k) &&\text{on}\quad \hat{\Gamma} \\
    \mathbb{S}_y(u,\varpi)\mathbf{n}-\sigma\mathbf{n}\Delta_\Gamma h
    &=\mathcal{J}(v,k)+\sigma\mathbf{n}\kappa &&\text{on}\quad \hat{\Gamma}
\end{aligned}
\end{equation}
for $(k,v)\in X_T\times Y_T$ and consider a mapping $\mathcal{S}$ from $(k,v)$ to $(h,u)$ by applying  Theorem~\ref{TShMR}.
 To prove Theorem~\ref{TLW}, we construct a self-mapping by taking a ball in $X_T\times Y_T$ of the form
\[
    \mathbb{B}=\left\{k\bigm|
    \|k\|_{X_T}\le\varepsilon'\right\}
    \times \left\{v\bigm|
    \|v\|_{Y_T}\le B'\right\}
\]
with small $\varepsilon'>0$ and small $T$, where $B'\le CB$ with some constant $C$.
The terms involving $k$ in \eqref{EMap} are small and contribution to $h$ and $u$ from the right-hand side is small for small $T$.
The whole process is tedious but it is possible.
 For the linear estimates in Theorem~\ref{TShMR}, we only invoke the case $L_1=0$, $L_2=0$ since we are just interested in local-in-time solvability and lower order linear terms are harmless.
 To estimate nonlinear terms $\mathcal{F}$ and $\mathcal{D}$, Lemma~\ref{LDif} with $h_2=0$ plays a key note.
 To estimate $\mathcal{E}$, we use estimate in Proposition~\ref{PDifA} only on $\hat{\Gamma}$.
 To estimate $\mathcal{J}_{\tan}$ and $\mathcal{J}_\mathrm{nor}$ we fully use Proposition~\ref{PDifA} with its time derivative estimates when $h_i$ depends on time (Lemma~\ref{Est:Ah:FBC}).
 To control the norm of its lift $\hat{\mathcal{J}}$ in $H_p^{1/2}\left(0,T;L^q(\Omega)\right)\cap L^p\left(0,T;W^{1,q}(\Omega)\right)$.
 In \cite{Sh}, the estimates for $\mathcal{A}_h$ and $\hat{\mathcal{A}}_h$ are not used so the proof is more involved.
The reason we use the lift as in \cite{Sh} is that we need to control pressure which is obtained by a kind of the Helmholtz decomposition.

Once we prove that $\mathcal{S}$ is a self-mapping in $\mathbb{B}$, i.e., $\mathcal{S}:\mathbb{B}\to\mathbb{B}$, it remains to prove that $\mathcal{S}$ is a strict construction by taking $T$ smaller.
 Then, Banach's fixed point theorem yields a unique fixed point of $\mathcal{S}$.
 This fixed point is exactly the solution what we seek for.
 In \cite{Sh}, the proof of strict contractivity of $\mathcal{S}$ is not explicit but we are able to prove by applying Lemma~\ref{LDif} and Proposition~\ref{PDifA} with its time derivative estimates for $\partial_t\hat{\mathcal{A}}_{h(t)}$ (Lemma~\ref{Est:Ah:FBC}).
 This is a brief idea of the proof of Theorem~\ref{TLW}.

%%%%%%
\subsection{Some technical issues}
\label{sub:techiss}

In the solvability of the linear Stokes system \eqref{EGS}, we note in particular that the free boundary data is given as a bulk quantity rather than a boundary quantity (cf.\ Theorem~\ref{TShMR}).
 This is because in Shibata's theory for the linear Stokes system, the estimate of the pressure is obtained by connecting the pressure to the fluid velocity and the height function that parameterizes the free boundary, through the weak Dirichlet operator which solves the Dirichlet boundary value problem for the Poisson equation weakly.
 This procedure corresponds to get a pressure by a kind of the Helmholtz decomposition.
Since the input for the weak Dirichlet operator should be functions defined on the entire domain $\Omega$, data defined only on the boundary $\Gamma$ cannot be input into the weak Dirichlet operator directly.
 To cope with this difficulty, Shibata considers the lifting of the boundary data, cf.\ \cite[Remark~2.6]{Sh}.
This finally results in the necessity that the free boundary data to be given as a bulk quantity in the solvability theorem for the linear Stokes system.
 As a result, when we try to estimate the free boundary condition in showing that the solution map $\mathcal{S}: \mathbb{B} \to \mathbb{B}$ is indeed a contraction map, we also have to consider the lifting of the boundary data to the bulk space.

Moreover, we estimate the time-dependence of the free boundary condition in the space of Bessel potentials.
 In \cite{Sh}, he only used the space on $\mathbb{R}$, apart from the lifting we also need to consider an extention in time for both the fluid velocity and the height function.
 This process may be avoided by considering the space of Bessel potentials on an interval as in \cite{PS}, but we follow Shibata's approach.
 This is why extensions in time for $u$ and $h$ are done in \cite[Section~6]{Sh}.

Let us explain his extension. 
 For a given function $f$ defined on $(0,T)$, we set the extension $f^\ast$ of $f$ in time to be
\begin{eqnarray*}
f^\ast(t) :=
\left\{
\begin{array}{lcl}
0 &\text{for}& t \leq 0, \\
f(t) &\text{for}& 0<t<T, \\
f(2T - t) &\text{for}& T<t<2T, \\
0 &\text{for}& t>2T.
\end{array}
\right.
\end{eqnarray*}
We pick $\psi \in C^\infty(\mathbb{R})$ that satisfies $\psi  \geq 0 $ in $\mathbb{R}$, $\psi = 1$ for $t > -1$ and $\psi = 0$ for $t < -2$.
With $\lambda$ being chosen sufficiently large, the Stokes system 
\begin{equation*}
\left\{
  \begin{aligned}
  \partial_t w + \lambda w - \operatorname{div} \mathbb{S}(w, \tau) &= 0& &\text{in}& &\Omega, \\
  \operatorname{div} w &= 0& &\text{in}& &\Omega, \\
  \partial_t \tau + \lambda \tau - w \cdot \mathbf{n} &= 0& &\text{on}& &\hat{\Gamma}, \\
  \mathbb{S}(w, \tau) \mathbf{n} - \sigma \Delta_{\hat{\Gamma}} \tau \mathbf{n}&= 0& &\text{on}& &\hat{\Gamma}, \\ 
  (w, \tau) \bigm|_{t=0} &= (0, h_0)& &\text{in}& &\Omega \times \hat{\Gamma}
  \end{aligned}
\right.
\end{equation*}
admits a unique solution $(w, \tau) \in X_T \times Y_T$; see e.g.\ \cite[Theorem~4.3]{Sh}.
Then, the extension $\widetilde{h}$ of $h$ in time can be constructed by
\begin{align} \label{Ext:time:h}
\widetilde{h}(t,\cdot) := h^\ast(t,\cdot) - \tau^\ast(t,\cdot) + \psi(t) \tau^\ast(|t|, \cdot), \quad t \in \mathbb{R}. 
\end{align}
Since $h(0,\cdot) - \tau(0,\cdot) = 0$, we note that $\widetilde{h}$ is differentiable with respect to $t$ for all $t \in \mathbb{R}$.
As a result, we can define the extension of $\mathcal{A}_h$ and $\mathcal{B}_h$, which is defined by formula \eqref{Ah:Bh}, in both time and spatial variable to be
\begin{align} \label{Lift:Aeta}
\hat{\mathcal{A}}_h := \chi\Big( \frac{d(\cdot, \hat{\Gamma})}{\delta_1} \Big) ( I - (L \widetilde{h}) \mathcal{L} \circ \Pi )^{-1} \nabla_{\hat{\Gamma}} ( (L \widetilde{h}) \circ \Pi ), \quad \hat{\mathcal{B}}_h := (1 + |\hat{\mathcal{A}}_h|^2)^{-\frac{1}{2}}
\end{align}
as is defined in \eqref{ELiftA}.
 We also define the lifting of $\mathbf{n}$ in spatial variable to be
\begin{equation}
    \hat{\mathbf{n}}:=\chi\left(\frac{d(\cdot,\hat{\Gamma})}{\delta_1}\right)
    \mathbf{n}\circ\Pi.
\end{equation} 
 On the other hand, the extension $\widetilde{u}$ of $u$ in time can be constructed as follows.
First, we pick a continuous extension $\overline{u}_0 \in B_{q,p}^{2(1 - 1/p)}(\mathbb{R}^d)$ of $u_0 \in B_{q,p}^{2(1 - 1/p)}(\Omega)$ in spatial variable $x$ that satisfies $\overline{u}_0 = u_0$ in $\Omega$ and set 
\begin{align*}
E(t) u_0 := \mathrm{e}^{- (2 - \Delta)t} \overline{u}_0 := \mathscr{F}^{-1}[ \mathrm{e}^{- (2 + |\xi|^2)t} \mathscr{F}[\overline{u}_0](\xi) ].
\end{align*}
The extension $\widetilde{u}$ of $u$ in time can then be constructed by
\begin{align*}
\widetilde{u}(t,\cdot) := u^\ast(t,\cdot) - [E(t) u_0]^\ast + \psi(t) [E(|t|) u_0]^\ast.
\end{align*}
Similarly, as $u_0 - E(0) u_0 = 0$, $\widetilde{u}$ is differentiable with respect to $t$ for all $t \in \mathbb{R}$.
Finally, we define $( \mathcal{F}(\widetilde{v}, \widetilde{k}), \mathcal{D}(\widetilde{v}, \widetilde{k}), \mathcal{E}(\widetilde{v}, \widetilde{k}),\mathcal{J}(\widetilde{v}, \widetilde{k}) )$ in the Stokes system \eqref{EMap} by replacing $(v, k, \mathcal{A}_k, \mathbf{n})$ in $( \mathcal{F}, \mathcal{D}, \mathcal{E}, \mathcal{J} )$ by $(\widetilde{v}, \widetilde{k}, \hat{\mathcal{A}}_k, \hat{\mathbf{n}})$.

To show that the solution map $\mathcal{S}: (v,k) \mapsto (u,h)$ is a contraction mapping, we mainly need two interpolation inequalities
\begin{align}
&\| v \|_{L^\infty ( 0,T; B_{q,p}^{2(1 - 1/p)}(\Omega) )} \notag \\
&\ \ \leq C \{ \| v_0 \|_{B_{q,p}^{2(1 - 1/p)}(\Omega) } + \| v \|_{L^p ( 0,T; W^{2,q}(\Omega) )} + \| \partial_t v \|_{L^p ( 0,T; L^q(\Omega) )} \} \label{IntPo1:Dom}
\end{align}
and
\begin{align}
&\| h \|_{L^\infty (0,T; B_{q,p}^{3 - 1/p - 1/q}(\hat{\Gamma}) )} \notag \\
&\ \ \leq C \{ \| h_0 \|_{B_{q,p}^{3 - 1/p - 1/q}(\hat{\Gamma}) }
+ \| h \|_{L^p (0,T; W^{3 - \frac{1}{q}, q}(\hat{\Gamma}) )} + \| \partial_t h \|_{L^p (0,T; W^{2 - \frac{1}{q}, q}(\hat{\Gamma}) )} \}\label{IntPo2:Bdy}
\end{align}
together with the inequality
\begin{align}
\| h \|_{L^\infty( 0,T; W^{2 - \frac{1}{q},q}(\hat{\Gamma}) )} \leq \| h_0 \|_{W^{2 - \frac{1}{q}, q}(\hat{\Gamma}) } + T^{\frac{1}{p'}} \| \partial_t h \|_{L^p ( 0,T; W^{2 - \frac{1}{q},q}(\hat{\Gamma}) )},\label{etr:inf:w2q}
\end{align}
cf.\ \cite[Section~6]{Sh}.
The estimate constants in \eqref{IntPo1:Dom} and \eqref{IntPo2:Bdy} depend only on the boundary regularity, and the inequality \eqref{etr:inf:w2q} can be easily derived by applying Minkowski's integral inequality to estimate the $W^{2 - \frac{1}{q}, q}$-norm of 
\begin{align*}
h(x,t) = h_0(x) + \int_0^t (\partial_s h)(x,s) \, ds \quad \text{where} \quad (t,x) \in (0, T) \times \hat{\Gamma}.
\end{align*}

In order to obtain the difference estimate for $\mathcal{J}$, we would need the following product rule.
Suppose that $1 < p < \infty$, $d<q<\infty$ and $0<T\leq 1$.
For 
\begin{align*}
g \in H_{q,p}^{1,\frac{1}{2}}(\Omega \times \mathbb{R}) := H_p^{\frac{1}{2}}( \mathbb{R}; L^q(\Omega) ) \cap L^p( \mathbb{R}; W^{1,q}(\Omega) )
\end{align*}
and $f \in W^{1,\infty}( \mathbb{R}; L^q(\Omega) ) \cap L^\infty( \mathbb{R}; W^{1,q}(\Omega) )$ that satisfies $\partial_t f \in L^p( (0,\infty); W^{1,q}(\Omega) )$ and $f$ vanishes for $t \notin (0, 2T)$, there exists a constant $C_{p,q}>0$, which depends only on $p$ and $q$, such that
\begin{align}
&\| f g \|_{H_{q,p}^{1,\frac{1}{2}}(\Omega \times \mathbb{R})} \notag \\
&\leq C_{p,q} \{ \| f \|_{L^\infty( \mathbb{R}; W^{1,q}(\Omega) )} + T^{\frac{q-d}{pq}} \| \partial_t f \|_{L^\infty( \mathbb{R}; L^q(\Omega) )}^{1 - \frac{d}{2q}} \| \partial_t f \|_{L^p( \mathbb{R}; W^{1,q}(\Omega) )}^{\frac{d}{2q}} \} \| g \|_{H_{q,p}^{1,\frac{1}{2}}(\Omega \times \mathbb{R})}, \label{ProR2:qp1h}
\end{align}
cf.\ \cite[Lemma~2.6]{ShSh07}.
Apart from this product rule, we need in addition an estimate for $\hat{\mathcal{A}}_h$ which reads as follows.

\begin{lem} \label{Est:Ah:FBC}
Suppose that $h_1, h_2 \in X_T$ satisfy the size condition in Theorem~\ref{TLW}. Then, there exist constants $C = C(p,q,d,r,\delta_0,K) > 0$ and $\mathcal{M} = \mathcal{M}(p,q,d,r,\delta_0,K,B) > 0$ such that if $\varepsilon + T^{\frac{1}{p'}} \mathcal{M} < (2 C)^{-1}$, the $\hat{A}_{h_i}$ defined by \eqref{Lift:Aeta} with $i = 1,2$ can be estimated by
\begin{align} \label{W1q:es:Ah}
\| \hat{A}_{h_i} \|_{L^\infty( \mathbb{R}; W^{1,q}(\Omega) )} \leq C (\varepsilon + T^{\frac{1}{p'}} \mathcal{M})
\end{align}
and
\begin{align} \label{Lq:es:dtAh}
\| \partial_t \hat{\mathcal{A}}_{h_i} \|_{L^\infty( \mathbb{R}; L^q(\Omega) )} \leq C \mathcal{M}
\end{align}
Moreover, the difference $\hat{A}_{h_1} - \hat{A}_{h_2}$ can be estimated by
\begin{align}
\| \hat{\mathcal{A}}_{h_1} - \hat{\mathcal{A}}_{h_2} \|_{L^\infty( \mathbb{R}; W^{1,q}(\Omega) )} \leq C T^{\frac{1}{p'}} \| \partial_t (h_1 - h_2) \|_{L^p(0,T; W^{2 - \frac{1}{q}, q}(\hat{\Gamma}) )} \label{difA:est1}
\end{align}
and
\begin{align}
&\| \partial_t \hat{\mathcal{A}}_{h_1} - \partial_t \hat{\mathcal{A}}_{h_2} \|_{L^\infty( \mathbb{R}; L^q(\Omega) )} \notag \\
&\ \ \leq C \{ \| \partial_t (h_1 - h_2) \|_{L^\infty ( 0,T; W^{1 - \frac{1}{q},q}(\hat{\Gamma}) )} + \| \partial_t (h_1 - h_2) \|_{L^p( 0,T; W^{2 - \frac{1}{q},q}(\hat{\Gamma}) )} \}. \label{Dif:dtAh}
\end{align}
\end{lem}
\begin{proof}
This lemma can be proved as in Proposition~\ref{PDifA}.
 Indeed, this lemma can be derived by combining the inequality \eqref{etr:inf:w2q} with the fact that $W^{1,q}(\Omega)$ and $W^{1 - \frac{1}{q},q}(\hat{\Gamma})$ are both Banach algebra provided that $d < q < \infty$.
\end{proof}

By \eqref{Ext:time:h}, we can observe that
\begin{align}
(L \widetilde{h}_1) (t, y) - (L \widetilde{h}_2) (t,y) = \widetilde{h}_1^\ast(t,y) - \widetilde{h}_2^\ast(t,y) = 0
\end{align}
for $t \notin (0, 2T)$ and $y \in \hat{\Gamma}$.
As a result, it can be deduced from the Neumann series expansion that $\hat{\mathcal{A}}_1(t, \cdot) - \hat{\mathcal{A}}_2(t, \cdot) = 0$ for $t \notin (0, 2T)$.
Hence, by combining Lemma~\ref{Est:Ah:FBC} with the inequality \eqref{ProR2:qp1h}, we can derive the difference estimate for $\mathcal{J}$.

\begin{rem} \label{Depen:T}
It should be emphasized that constants $\varepsilon, B$ in Lemma~\ref{Est:Ah:FBC} are respectively the size restriction for initial height and initial velocity in Theorem~\ref{TLW}.
Moreover, the constant $\mathcal{M}$ in Lemma~\ref{Est:Ah:FBC} is a quadratic polynomial in $B$ with coefficients being independent of $\varepsilon$, $h_0$ and $T$, cf.\ \cite[Section~6]{Sh}.
Hence, from the size condition $\varepsilon + T^{\frac{1}{p'}} \mathcal{M} < (2 C)^{-1}$, we can observe the philosophy that the time for local existence is independent of $h_0$.
\end{rem}
%
%%%%%%

%%%%%%%%%%%%%%%%%%%%%%%%%%%%%%%%%%%%%%%%%%%
\section*{Acknowledgements}

This work was done as a part of research activities of Social Cooperation Program ``Mathematical Science for Refrigerant Thermal Fluids" sponsored by Daikin Industries, Ltd.\ at the University of Tokyo.
 The authors are grateful to members of the Technology and Innovation Center of Daikin Industries, Ltd.\ for showing several interesting phenomena related to collapse of a bubble with fruitful discussion which triggered this work.
 The work of the first author was partly supported by the Japan Society for the Promotion of Science (JSPS) through the grants KAKENHI: JP24K00531, JP24H00183 and by Arithmer Inc., Daikin Industries, Ltd.\ and Ebara Corporation through collaborative grants.
 The work of the second author was partly supported by Daikin Industries, Ltd.\ through collaborative grants.
 This work is mainly developed during the stay of the second author as a postdoc at the Graduate School of Mathematical Sciences at the University of Tokyo.
 Its hospitality is gratefully acknowledged.
 
%
%%%%%%%%%%%%%%%%%%%%%%%%%%%%%%%%%%%%%%%%%%%

\bigskip

\noindent
(Y.~Giga)
{\it Email address}: \href{mailto:labgiga@ms.u-tokyo.ac.jp}{\nolinkurl{labgiga@ms.u-tokyo.ac.jp}}\\
(Z.~Gu)
{\it Email address}: \href{mailto:guzy@szu.edu.cn}{\nolinkurl{guzy@szu.edu.cn}}

\end{document}